\documentclass[11pt, leqno]{amsart}
\usepackage{amsfonts,amssymb,amscd,amsmath,amsthm}
\usepackage{tikz}
\usetikzlibrary{decorations.fractals}
\usepackage[nice]{nicefrac}
\usepackage{lmodern}
\usepackage{mathrsfs}
\usepackage{qsymbols}
\usepackage{mathtools}
\mathtoolsset{showonlyrefs}
\usepackage{dsfont}
\usepackage{graphicx}
\usepackage{latexsym}
\usepackage{chngcntr}
\usepackage{paralist}
\usepackage[parfill]{parskip}
\usepackage{bm}
\usepackage{enumitem}
\usepackage{tikz-cd}
\usepackage{csquotes}
\usepackage{todonotes}
\usepackage[plainpages=false,pdfpagelabels,
citecolor=red]{hyperref}
\usepackage{xcolor}
\hypersetup{
	colorlinks,
	linkcolor={cyan!90!black},
	citecolor={magenta},
	urlcolor={green!40!black}}
\let\oldproofname=\proofname
\renewcommand{\proofname}{\rm\bf{\oldproofname}}
\renewcommand{\emph}[1]{{\rm\bf{#1}}}
\newtheorem{theorem}{Theorem}[section]
\newtheorem{lemma}[theorem]{Lemma}
\newtheorem{proposition}[theorem]{Proposition}
\newtheorem{corollary}[theorem]{Corollary}
\newtheorem{assumption}[theorem]{Assumption}
\theoremstyle{definition}
\newtheorem{definition}[theorem]{Definition}
\newtheorem{remark}[theorem]{Remark}

\numberwithin{equation}{section}
\newcommand{\loc}{\mathrm{loc}}
\DeclareMathOperator{\e}{e}
\DeclareMathOperator{\Div}{div}
\DeclareMathOperator{\sgn}{sgn}

\DeclareMathOperator{\dist}{dist}
\DeclareMathOperator{\cc}{c}

\newcommand{\IL}{\mathbb{L}}
\DeclareMathOperator{\tr}{Tr}
\newcommand{\ext}{\mathrm{E}}
\newcommand{\energy}{\mathscr{E}}
\newcommand{\1}{\boldsymbol{1}}
\newcommand{\C}{\mathbb{C}}
\newcommand{\N}{\mathbb{N}}

\newcommand{\R}{\mathbb{R}}

\newcommand{\smooth}[1][]{\rC_{\cc}^{\infty}(#1)}
\newcommand{\smoothD}[1][]{\rC_{D}^{\infty}(#1)}

\renewcommand{\H}{\mathrm{H}}

\newcommand{\W}{\mathrm{W}}

\renewcommand{\L}{\mathrm{L}}

\DeclareMathOperator{\D}{dom}

\DeclareMathOperator{\Ker}{ker}
\renewcommand{\d}{\mathrm{d}}

\DeclareMathOperator{\esssup}{esssup}
\DeclareMathOperator{\essinf}{essinf}
\newcommand{\glob}{\mathrm{glob}}
\newcommand{\inter}{\mathrm{int}}

\renewcommand{\S}{\mathrm{S}}

\renewcommand{\i}{\mathrm{i}}
\DeclareMathOperator{\im}{Im}
\DeclareMathOperator{\re}{Re}
\newcommand{\rB}{\mathrm{B}}
\newcommand{\rC}{\mathrm{C}}

\newcommand{\cH}{\mathcal{H}}

\newcommand{\cL}{\mathcal{L}}

\newcommand{\LL}{\mathscr{L}}
\newcommand{\sub}{\subseteq}
\def\Xint#1{\mathchoice
	{\XXint\displaystyle\textstyle{#1}}%
	{\XXint\textstyle\scriptstyle{#1}}%
	{\XXint\scriptstyle\scriptscriptstyle{#1}}%
	{\XXint\scriptscriptstyle%
		\scriptscriptstyle{#1}}%
	\!\int}
\def\XXint#1#2#3{{\setbox0=\hbox{$#1{#2#3}{%
				\int}$ }
		\vcenter{\hbox{$#2#3$ }}\kern-.6\wd0}}
\def\barint{\,\Xint-} %\, corrects the \! used in the definition
\title[Functional calculus and dynamical boundary conditions]{Bounded functional calculus for divergence form operators with dynamical boundary conditions}
\author{Tim B\"ohnlein}
\author{Moritz Egert}
\email{boehnlein@mathematik.tu-darmstadt.de, egert@mathematik.tu-darmstadt.de}
\address{Fachbereich Mathematik, Technische Universit\"at Darmstadt, Schlossgartenstr.~7, 64289 Darmstadt, Germany}
\author{Joachim Rehberg}
\email{joachim.rehberg@wias-berlin.de}
\address{Weierstrass Institut, Mohrenstr.~39, 10117 Berlin, Germany}

\date{\today}
\dedicatory{}
\begin{document}
	\begin{abstract}
		We consider divergence form operators with complex coefficients on an open subset of Euclidean space. Boundary conditions in the corresponding parabolic problem are dynamical, that is, the time derivative appears on the boundary. As a matter of fact, the elliptic operator and its semigroup act simultaneously in the interior and on the boundary. We show that the elliptic operator has a bounded $\H^\infty$-calculus in $\L^p$ if the coefficients satisfy a $p$-adapted ellipticity condition. A major challenge in the proof is that different parts of the spatial domain of the operator have different dimensions. Our strategy relies on extending a contractivity criterion due to Nittka and a non-linear heat flow method recently popularized by Carbonaro--Dragičević to our setting. 
	\end{abstract}
	
	\maketitle
	
	\section{Introduction}
	This paper is dedicated to elliptic operators in divergence form along with their associated parabolic problems, which are commonly encountered in science. These scenarios frequently pose challenges due to discontinuous coefficient functions and singular objects on the right-hand side, which reside on sets with Hausdorff dimension smaller than the spatial dimension. It is widely recognized, e.g.\ in the theory of electricity (see the monograph of the Nobel prize winner I.~Tamm \cite[Chap.~1.4]{Tamm_book}) that (spatial) jumps in the coefficient function are intimately connected to the presence of surface densities on the right-hand side. 

    \subsection{The parabolic equation}
	First, let us give a formal description of the linear parabolic problem with dynamical boundary conditions that we have in mind. An excellent exposition for the derivation of such equations in sciences can be found in~\cite{Goldstein_Derivation}. In dimension $d \geq 2$ we let $O \sub \R^d$ be open, $D \sub \partial O$, $\Sigma \sub \overline{O} \setminus D$, $T \in (0, \infty]$ and consider the system
	\begin{align} \label{eq: a-formalGleichung}
		\partial_t u - \Div (A \nabla u) &= f|_{O \setminus \Sigma} &\text{in} \; (0, T) \times (O \setminus \Sigma), \nonumber \\
		\partial_t u + \nu \cdot A \nabla u &= f|_{\Sigma \cap \partial O} &\text{on} \; (0, T) \times (\Sigma \cap \partial O), \nonumber \\
		\partial_t u + \nu_{\Sigma \cap O} \cdot A \nabla u &= f|_{\Sigma \cap O} &\text{on} \; (0, T) \times (\Sigma \cap O), \\
		u &=0 &\text{on} \; (0, T) \times D, \nonumber \\
		\nu \cdot A \nabla u &= 0 &\text{on} \; (0, T) \times \partial O \setminus (D \cup \Sigma), \nonumber \\
		u(0) &= u_0 &\text{in} \; O \cup \Sigma. \nonumber
	\end{align}
	Here, $A$ is a \textbf{uniformly strongly elliptic} coefficient function with complex, bounded and measurable entries, that is,
	\begin{equation} \label{eq: Intro: ellipticity}
		\lambda(A) \coloneqq \underset{x \in O}{\essinf} \min_{|\xi| =1} \re ( A(x) \xi \cdot \overline{\xi}) > 0 \quad \& \quad \Lambda(A) \coloneqq \underset{x \in O}{\esssup} |A(x)| < \infty, 
	\end{equation}
	the vector $\nu$ is the outer unit normal, $\nu_{\Sigma \cap O}$ denotes a `jump' in the normal derivative on $\Sigma \cap O$, and $f, u_0$ are functions defined on $O \cup \Sigma$. Hence, the underlying set for the dynamics is $O \cup \Sigma$, where the `volume' $O$ is equipped with the $d$-dimensional Lebesgue measure, but the `surface' $\Sigma$ is a Lebesgue null set and carries a different Radon measure $m$. In science, $\Sigma$ would typically be a Lipschitz surface with $(d-1)$-dimensional Hausdorff measure, but our mathematical treatment allows it to be as wild as the von Koch snowflake with its natural measure of fractal dimension, or of Hausdorff co-dimension larger than $1$.

    \subsection{The variational setting}
	Following \cite{Arendt_Sauter_Wentzell_Laplace_and_Appr_Trace, Disser-Meyries-Rehberg_Parabolic, EMR-DBC, Nittka}, we model \eqref{eq: a-formalGleichung} as an abstract Cauchy problem 
	\begin{align}
		\begin{split}  \label{eq: Intro: ACP in L2}
			u'(t) + \LL^A u(t) &= f(t) \qquad (t \in (0, T)),
			\\ u(0) &= u_0, 
		\end{split}
	\end{align}
	in the natural $\L^2$-space $\IL^2 \coloneqq \L^2(O \cup \Sigma, \d x \oplus \d m)$ that takes the two dynamical parts into account. This can efficiently be done by an extension of the form method
	due to Arendt--ter Elst~\cite{Arendt_Elst_j-elliptic_forms}, see Section~\ref{Sec: L2-theory} for details. In this construction, $\LL^A$ is associated with the usual sesquilinear form for divergence form operators, 
	\begin{align} \label{eq: Intro: a}
		a\colon V \times V \to \C, \quad a(u,v) \coloneqq \int_O A \nabla u \cdot \overline{\nabla v} \, \d x,
	\end{align}
	where $V = \W^{1,2}_D(O)$ is a Sobolev space that models the homogeneous Dirichlet condition on $D$ in the fourth line of \eqref{eq: a-formalGleichung}. However, $V$ is not considered as a subspace of $\L^2(O)$ but of $\IL^2$ through the identification operator $j(u) = u|_O \oplus \tr(u)$ in order to account for the dynamics on $\Sigma$, compare with \cite[Sec.~1]{EMR-DBC}. To this end, we need minimal geometric assumptions that we describe next.

	\begin{assumption} \label{Ass: DBC}
		
		Throughout the entire paper, $d \geq 2$ and:
		
		\begin{enumerate}
			
			\item The set $O \sub \R^d$ is open and non-empty.
			
			\item The \emph{Dirichlet part} $D \sub \partial O$ is closed and possibly empty.
			
			\item\label{Ass: DBC: sigma} The \emph{dynamical part} $\Sigma \sub \overline{O} \setminus D$ is a non-empty Borel set with Lebesgue measure zero.
			
			\item\label{Ass: DBC: trace} The \emph{trace} to $\Sigma$, defined for $u \in V \cap \rC(\overline{O})$ by pointwise restriction
			\begin{align*}
				\tr(u) \coloneqq u|_{\Sigma},
			\end{align*}
			extends by density to a bounded linear operator from $V$ into $\L^2(\Sigma, \d m)$, where $m$ is a Radon measure on $\Sigma$.
		\end{enumerate}
	\end{assumption}
	
	\begin{remark} \label{Rem: Ass: DBC}
		We denote the extension of the trace map in Assumption~\ref{Ass: DBC}~\ref{Ass: DBC: trace} by the same symbol. As usual, a Radon measure is a Borel measure that is finite on (relative) compact sets, outer regular on Borel sets, and inner regular on (relative) open sets.
	\end{remark}

 	We refer to Section~\ref{Sec: Geometry} for explicit examples that we have in mind, including fractal sets $\Sigma$, but it is conceptually much simpler to stick to the more general assumption above.

    \subsection{Main results}
	
	Abstract theory of sectorial forms tells us that $\LL^A$ generates a holomorphic contraction semigroup in $\IL^2$. Thus, $\LL^A$ has a bounded $\H^\infty$-calculus in $\IL^2$ and notably exhibits maximal parabolic regularity, serving as a powerful tool for handling nonlinear versions of \eqref{eq: a-formalGleichung} via fixed-point methods, and even stochastic maximal regularity~\cite{vanNeerven}.

    The primary focus of this work is to elaborate whether these properties extend to $\IL^p \coloneqq \L^p(O \cup \Sigma, \d x \oplus \d m)$ for $p \neq 2$. For operators with \emph{real coefficients} this problem and related ones are extensively investigated in the full Lebesgue scale, see e.g.\ \cite{Disser-Meyries-Rehberg_Parabolic, EMR-DBC, Goldstein_DynBC, Hoemb_Control}. Their techniques do not carry over to the complex case, which, to the best of our knowledge, has  been an open problem. We settle the question in our main result:
	
	\begin{theorem} \label{Intro: Thm: Bounded FC in Lp}
		Let $p \in (1, \infty)$ and suppose that $A$ is $p$-elliptic. The semigroup $(\e^{-t \LL^A})_{t \geq 0}$ extends to an analytic contraction semigroup on $\IL^p$ of angle $\theta_p \in (0, \nicefrac{\pi}{2})$ defined in \eqref{eq: Intro: theta_p} and its generator $\LL_p^A$ has a bounded $\H^{\infty}$-calculus of angle $\nicefrac{\pi}{2} - \theta_p$. In particular, $\LL_p^A$ has maximal parabolic regularity.
	\end{theorem}

	Above, \emph{$\boldsymbol{p}$-ellipticity} refers to an algebraic condition on the coefficients that originates from Cialdea--Maz'ya \cite{CM_p-elliptic}: 
	\begin{equation} \label{eq: Intro: p-ellipticity}
		\Delta_p(A) \coloneqq \underset{x \in O}{\essinf}  \min_{|\xi| =1} \re \Big( A(x) \xi \cdot \overline{(\xi + (1 - \nicefrac{2}{p}) \overline{\xi})} \Big)> 0.
	\end{equation}
	This notion was introduced by Carbonaro--Dragičević \cite{CD_JEMS} and independently by Dindoš--Pipher \cite{DP_p-elliptic}. It bridges between uniform strong ellipticity for complex matrices ($p=2$) and real matrices ($p=\infty$), compare with Section~\ref{Subsec: p-ellipticity}. Every uniformly strongly elliptic $A$ is $p$-elliptic in a range of $p$'s that depends on $\lambda(A), \Lambda(A)$ \cite[Cor.~3]{Egert_p-ellipticity}. Therefore, Theorem~\ref{Intro: Thm: Bounded FC in Lp} always applies in an $A$-dependent range of $p$'s. The angle above is
	\begin{equation} \label{eq: Intro: theta_p}
		\theta_p \coloneqq \sup \Big\{ \theta\in [0, \nicefrac{\pi}{2}) \colon \e^{\pm \i \theta} A \; \text{is} \; p\text{-elliptic} \Big\}
	\end{equation}
	and we address the problem of finding lower bounds for $\theta_p$ in terms of the `data' of $A$ in Section~\ref{Sec: Explicit angle}.
	
	Compared to many $\L^p$-extrapolation results related to elliptic operators \cite{Auscher_Memoirs, Bechtel_Lp, BK1, BK2, Böhnlein-Egert_Lp_estimates, Egert_p-ellipticity, P-ellipticity-Moritz-Counterpart, Shen_Lp_extrapolation}, there is no clear dimensional scaling on the spatial domain of our operators due to the presence of $\Sigma$. This seems to forbid any use of (generalized) kernel estimates on objects related to $\LL^A$. In proving Theorem~\ref{Intro: Thm: Bounded FC in Lp}, we had to watch out for methods that predominantly work on the level of the sesquilinear form $a$ in \eqref{eq: Intro: a} and not on the associated operator $\LL^A$, because the former one does not involve $\Sigma$. We found a suitable approach in the non-linear heat flow technique of Carbonaro--Dragičević \cite{CD_Bakry_estimate, CD_JEMS, CD_Open_Set, CD_Schrodinger, CD_Trilinear_embedding}. Largely inspired by their results, we first prove a \emph{bilinear embedding} in Section~\ref{Sec: BE}:
	
	\begin{theorem}  \label{Intro: Thm: BE}
		Let $p \in (1, \infty)$ and $A, B$ be $p$-elliptic. Then there is a constant $C > 0$ that depends only on $p$, $\lambda(A,B)$, $\Lambda(A,B)$ and $\Delta_p(A,B)$ such that
		\begin{equation}  \label{eq: BFC on open sets, Bilinear estimate}
			\int_0^{\infty} \int_{O} \big|\nabla (\e^{- t \LL^A}f ) |_O\big| \cdot \big|\nabla (\e^{-t \LL^B}g)|_O\big| \, \d x \, \d t \leq C \Vert f \Vert_p \Vert g \Vert_{p'} \qquad (f,g \in \IL^p \cap \IL^{p'}).
		\end{equation}
	\end{theorem}
	
	Here, $p'= \nicefrac{p}{(p-1)}$ denotes the Hölder conjugate of $p$ and 
	\begin{equation*}
		\lambda(A,B) \coloneqq \lambda(A) \land \lambda(B), \quad \Lambda(A,B) \coloneqq \Lambda(A) \lor \Lambda(B) \quad \& \quad \Delta_p(A,B) \coloneqq \Delta_p(A) \land \Delta_p(B) 
	\end{equation*}
	are common ($p$-)ellipticty constants for $A$ and $B$. Theorem~\ref{Intro: Thm: BE} implies Theorem~\ref{Intro: Thm: Bounded FC in Lp} by standard quadratic estimates due to Cowling--Dust--McIntosh--Yagi \cite{CDMcIY_BFC}, see Section~\ref{Sec: Bounded FC in Lp}. However, this isn't the whole story, as the boundedness of the semigroup in $\IL^p$ is needed in the proof of Theorem~\ref{Intro: Thm: BE}. We will address this prerequisite beforehand in Section~\ref{Sec: Lp-contractivity} by providing a generalization of Nittka's invariance criterion \cite{Nittka} tailored to our needs.
	
	In both, the bilinear embedding and Nittka's criterion, the main novelty is the presence of the identification operator $j$. It poses new technical challenges that we resolve in this work. We make essential use of the particular choice of $j$, or more precisely, that it is injective and that $j$ as well as $j^{-1}$ commute with certain non-linear maps.

    \subsection{\texorpdfstring{Discussion of the $\boldsymbol{\L^p}$-setting}{}}

    Let us provide further motivation for considering the elliptic/parabolic operators in the particular spaces $\IL^p$.
     
	In the realm of mathematical semiconductor modeling, a widely studied model is the Van Roosbroeck system \cite{Markowich_book, Selberherr_book, Meinlschmidt-Rehberg_Van-Roosbroeck}. This system comprises a set of nonlinear drift-diffusion equations with surface charge densities on the right-hand sides. Mathematically, they are understood as measures concentrated on surfaces. In an advanced writing of the system \cite{Gajewski_Semiconductor}, dynamical boundary conditions emerge on parts of the boundary. Solving the system numerically is a nontrivial task and the only known successful algorithm, due to Scharfetter--Gummel, bases on two essentials. First, if the system is studied in some function space, then its dual has to contain indicator functions of subsets, such as boxes or tetrahedra, as ‘test functions’. This is true in $\IL^p$ and false in the more general distribution spaces used e.g.\ in \cite{Meinlschmidt-Rehberg_Van-Roosbroeck}. Second, once having tested the system with indicator functions on a partition of subdomains, a point balance across all subdomains needs to be established. Usually, this conversion involves transforming local volume integrals into surface integrals utilizing Gau\ss’ theorem \cite{Comi-Payne_Gauss}. While this method is effective when the flux's divergence is a measure, it fails when it is only a distribution. 

    However, already when dealing with \emph{semilinear} parabolic equations, a standard approach to handling right-hand sides $R$ that depend nonlinearly on the solution, is based on the fact that $R$ is (locally) Lipschitz continuous with respect to the solution. This holds when the solution space is equipped with the topology of an interpolation space between the underlying Banach space and the domain of the elliptic operator~\cite[Chap.~6]{Pazy_book}. For parabolic equations on $O$ it typically suffices to know that domains of fractional powers of the linear elliptic operator embed into $\L^\infty$ in order to catch the nonlinearities~\cite{Elst-Rehberg_Linfty-estimates-for-L} but in the case of the $\IL^p$ spaces involving two different measures, a pointwise uniform control is needed. 
        
    We address this issue in Section~\ref{Sec: Elliptic regularity for DBC}, where for large $p$ we establish embeddings of fractional power domains of $\LL_p^A$ into spaces of H\"older continuous functions on $O \cup \Sigma$ in various settings. These embeddings generalize results in \cite{Disser-Elst-Rehberg_Hölder, Disser-Meyries-Rehberg_Parabolic, EMR-DBC, H-J-K-R_Ell-Par-Reg-MBC, Hoppe-Meinlschmidt-Neitzel_QLparPDEs-MBC} for parabolic equations on $O$ and open the door for proving H\"older regularity for solutions of the parabolic equations simultaneously in space and time, which can be very useful in applications, see e.g.\ \cite{Betz_Optimal-Control, Casas-Yong_Optimal-control} and references therein. The proof relies on delicate yet known mapping properties of the sesquilinear form \eqref{eq: Intro: a}, which are independent of the dynamical boundary conditions, and a `transference formula' for the inverse of $\LL^A$, which might be of independent interest. 

    \subsection{Notation}
	We write $B(x,r) \sub \R^d$ for the open ball with center $x \in \R^d$ and radius $r > 0$ and denote inner products by $\langle \cdot\,, \cdot \rangle$ if the context is clear. The (almost everywhere) restriction of functions $u$ to a subset $F$ is indicated by $u|_F$. 

  % and opens the door for (H\"older) regularity results of linear and semilinear parabolic equations, see \todo{References}. The latter is particularly useful in applications, see e.~g.\ \cite{Casas}, \cite{Betz} and references therein.\todo{References} In the static setting it suffices to know that domains of fractional powers embed into $\L^\infty(O)$ to catch the nonlinearities, see \cite{ElstJo}. On the other hand, the dynamical setup demands for the stronger embedding into some H\"older space in order to infer H\"older regularity of the solution. We investigate such embeddings in Section~\ref{Sec: Elliptic regularity for DBC} by the following strategy: we use a `transference principle' (see Lemma~\ref{Formula for L via L-M: Lem: Repres via j}) to transport results for the `static' Lax--Milgram operator to the dynamical setting. The main advantage of this approach is that we can invoke deep facts from harmonic analysis that are not available in the dynamical setup. 
	
	\section{Geometry} \label{Sec: Geometry}
	
	This section contains prerequisites on function spaces and a unified treatment of Sobolev traces, leading to a variety of geometric configurations that match with our background assumption (Assumption~\ref{Ass: DBC}).
	
	\begin{definition}
		Let $\smooth[\R^d \setminus D]$ be the space of all $\C$-valued, smooth and compactly supported functions on $\R^d$ whose support avoids $D$. For $F \sub \R^d$ we define
		\begin{equation*}
			\smoothD[F] \coloneqq \smooth[\R^d \setminus D]|_F.
		\end{equation*}
	\end{definition}
	
	\begin{definition} \label{Geometry: Def: W^{1,p}_D and dual}
		Let $p \in [1, \infty)$. We denote by $\W_D^{1,p}(O)$ the closure of $\smoothD[O]$ with respect to the norm $\| \cdot \|_{1,p} \coloneqq ( \| \cdot \|_p^p + \| \nabla \cdot \|_p^p)^{\nicefrac{1}{p}}$ and abbreviate $V \coloneqq \W_D^{1,2}(O)$. Furthermore, we write $\W^{-1, p}_D(O) \coloneqq (\W^{1, p'}_D(O))^*$, the space of bounded anti-linear functionals on $\W^{1,p'}_D(O)$. 
	\end{definition}
	
	By uniform continuity, all functions in $\smoothD[O]$ extend continuously to $\overline{O}$. Hence, $V \cap \rC(\overline{O})$ is dense in $V$ and Assumption~\ref{Ass: DBC}~\ref{Ass: DBC: trace} makes sense. Note that we use $\smash{\W^{1,2}_\varnothing(O)}$ when the Dirichlet part is empty, which should be thought of as a regularized version of the usual Sobolev space $\W^{1,2}(O)$ that contains enough functions with a well-defined trace, compare with \cite[Chap.~4]{Ouhabaz_book} and \cite{Arendt_Elst_D_to_N, Arendt_Sauter_Wentzell_Laplace_and_Appr_Trace, Sauter_Uniqueness_of_Appr_Trace}. 
	
	\subsection{The identification map \texorpdfstring{\boldmath$j$}{j}} 
	
	To define dynamical boundary conditions rigorously via the form method, we use an embedding of $V$ into an $\L^2$-space on $O \cup \Sigma$. Below, $m$ is the Radon measure from Assumption~\ref{Ass: DBC}~\ref{Ass: DBC: trace}.
	
	\begin{definition} \label{Traces: Def: The map j and Lp over the hybrid measure space}
		Let $p \in [1, \infty]$. We denote by $\L^p(O) = \L^p(O, \d x)$ the usual Lebesgue space of $p$-integrable functions, put $\d \mu = \d x \oplus \d m$ and write
		\begin{equation*}
			\IL^p \coloneqq \L^p (O \cup \Sigma, \d \mu) \coloneqq \L^p(O) \oplus \L^p(\Sigma).
		\end{equation*}
		We refer to this space as the $\L^p$-space over the \emph{hybrid measure space} $(O \cup \Sigma, \d \mu)$. The \emph{identification operator} is given by
		\begin{equation*}
			j \colon V \to \IL^2, \quad j(u) \coloneqq u|_O \oplus \tr(u).
		\end{equation*}
	\end{definition}
	
	The terminology `hybrid' highlights the fact that parts of such functions `live' on the volume $(O, \d x)$, whereas another part `lives' on the typically lower dimensional set $(\Sigma, \d m)$. We could have also used $O \setminus \Sigma$ for the volume part in order to have `restrictions' to two disjoint sets. However, there's no confusion here as $\Sigma$ is assumed to be a Lebesgue null set according to Assumption~\ref{Ass: DBC}~\ref{Ass: DBC: sigma}. 
	
	\begin{lemma}  \label{Traces: Lem: j(V) dense}
		The space $\smoothD[\Sigma]$ is dense in $\L^2(\Sigma)$. In particular, $j(V)$ is dense in $\IL^2$.
	\end{lemma}
	
	\begin{proof}
		Density of $\smoothD[\Sigma]$ in $\L^2(\Sigma)$ follows by combining three density results. 
		First, by dominated convergence, $\L^2(\Sigma)$-functions with support in a compact set $\Sigma' \sub \Sigma$ with $\dist(\Sigma',D)>0$ are dense in $\L^2(\Sigma)$. Second, $\Sigma'$ carries the Radon measure $m|_{\Sigma'}$, so $\rC(\Sigma')$ is dense in $\L^2 (\Sigma', \d m)$, see for instance \cite[Thm.~19.38]{Yeh_Real-Analysis}. Third, $\Sigma'$ keeps a positive distance to $D$, so the restrictions
		$\smoothD[\Sigma']$ form a unital $\ast$-algebra that separates the points of $\rC(\Sigma')$ and the Stone--Weierstrass theorem yields that $\smoothD[\Sigma']$ is $\| \cdot \|_{\infty}$-dense in $\rC(\Sigma')$. 
		
		In order to see that the above implies that $j$ has dense range in $\IL^2$, we suppose that $u \in j(V)^\perp$. Then, we have
		\begin{align*}
			0 = \langle u|_O, v|_O \rangle + \langle u|_\Sigma, v|_\Sigma \rangle \qquad (v \in j(V) \sub \IL^2).
		\end{align*}
		The second term on the right vanishes if $v \in \smooth[O]$ vanishes on $\Sigma \cap O$. Since $\Sigma$ is a Lebesgue null set, such functions are dense in $\L^2(O)$ and we conclude that $u|_O = 0$. Now, we can use the density result from the first part to conclude $u|_\Sigma = 0$ and thus $u=0$.
	\end{proof}
	
	For the $\IL^p$-theory, we shall need to commute $j$ and its inverse with certain non-linear maps. While this is clear for $j^{-1}$, because applying $j^{-1}$ means restricting functions in $j(V)$ from $O \cup \Sigma$ to $O$, the argument for $j$ is more involved.
	
	\begin{lemma}  \label{Traces: Lem: j commutes!}
		Let $k, l \in \N$, $\Phi \colon \C^k \to \C^l$ be Lipschitz continuous with $\Phi (0) = 0$ and let $U \coloneqq (u^i)_{i=1}^k \in V^k$. Then $\Phi(U) \in V^l$ and$$(j(\Phi(U)^1), \dots, j(\Phi(U)^l)) = \Phi(j(u^1), \dots, j(u^k)).$$
	\end{lemma}
	
	\begin{proof}
		We follow the argument in \cite[Lem.~4]{Egert_p-ellipticity}. Let $(U_n)_n = ((u_n^i)_{i=1}^k)_n \sub \smoothD[O]^k$ be such that $U_n \to U$ in $V^k$ as $n \to \infty$. Since $\Phi$ is Lipschitz continuous with $\Phi(0) =0$, it follows that $(\Phi(U_n))_n \sub V^l$ is bounded. Hence, we find some $v \in V^l$ such that $\Phi(U_n) \to v$ weakly in $V^l$ along a subsequence. On the other hand, $\Phi(U_n) \to \Phi(U)$ in $\L^2(O)^l$ and we conclude that $\Phi(U) = v \in V^l$. This proves the first assertion.
		
		To show the second one, we use the first part and that $j \colon V \to \IL^2$ is continuous to conclude that $(j(\Phi(U_n)^1), \dots, j(\Phi(U_n)^l)) \to (j(\Phi(U)^1), \dots, j(\Phi(U)^l))$ weakly in $(\IL^2)^l$ along a subsequence. Since all $\Phi(U_n)^i$ are continuous on $\overline{O}$, the map $\tr$ acts as an honest pointwise restriction and we can commute
		\begin{equation*}
			(j(\Phi(U_n)^1), \dots, j(\Phi(U_n)^l)) = \Phi(j(u_n^1), \dots, j(u_n^k)).
		\end{equation*}
		The right-hand side tends to $\Phi(j(u^1), \dots, j(u^k))$ in $\IL^2$ and the proof is complete.
	\end{proof}
	
	\subsection{Lebesgue points and traces} \label{Subsec: Lebesgue points and traces}
	
	For a globally defined function $f \in \L^1_{\loc}(\R^d)$, we recall that $x \in \R^d$ is a \emph{Lebesgue point} of $f$ if there exists $z \in \C$ such that
	\begin{equation*}
		\lim_{r \to 0}  \barint_{B(x,r)} |f(y) - z| \, \d y = 0,
	\end{equation*}
	compare with \cite[Sec.~6.2]{Hedberg}. In this case, we also have
	\begin{equation*}
		z = \lim_{r \to 0}  \barint_{B(x,r)} f(y) \, \d y.
	\end{equation*}
	Lebesgue points allow us to assign pointwise values to equivalence classes of functions, sometimes called `precise' or `refined' representative.
	
	\begin{definition}  \label{Traces: Def: Global trace}
		Let $f \in \L^1_{\loc}(\R^d)$ and $F \sub \R^d$. The \textbf{global trace} of $f$ to $F$ is defined as
		\begin{equation*}
			\tr_{\glob,F}(f)(x) \coloneqq \lim_{r \to 0}  \barint_{B(x,r)} f(y) \, \d y,
		\end{equation*}
		for all $x \in F$ for which the limit exists.
	\end{definition}
	
	If the set $F$ happens to be a subset of $\overline{O}$, we can also define a trace for functions that are only defined on $O$ as follows.
	
	\begin{definition}  \label{Traces: Def: Interior trace}
		Let $f$ be integrable on bounded subsets of $O$ and $F \sub \overline{O}$. The \textbf{interior trace} of $f$ to $F$ is defined as
		\begin{equation*}
			\tr_{\inter, F}(f)(x) \coloneqq \lim_{r \to 0} \barint_{O \cap B(x,r)} f(y) \, \d y,
		\end{equation*}
		for all $x \in F$ for which the limit exists.
	\end{definition}
	
	Whether global and interior traces exist in a suitable sense has been investigated extensively \cite{Arendt_Sauter_Wentzell_Laplace_and_Appr_Trace, Egert_Tolksdorf_Sobolev_Traces, Jonsson-Wallin, Sauter_Uniqueness_of_Appr_Trace, Swanson_Ziemer}. Here, our focus lies on making `soft' assumptions on $V$ that are commonly used in heat kernel theory on domains~\cite[Sec.~6.3]{Ouhabaz_book} rather than relying on geometric measure theory. By an \emph{extension} of a function $f$ on $O$ we mean a function $\ext (f)$ on $\R^d$ with the property that $(\ext (f))|_O = f$. An \emph{extension operator} $\ext \colon V \to \W^{1,2}(\R^d)$ is a linear operator such that $\ext (u)$ is an extension of $u$ for every $u \in V$.
	
	\begin{definition}  \label{Traces: Def: Extension property of V}
		We say that $V$ has the
		\begin{enumerate}
			\item \emph{embedding property} if there is some $\theta \in [0,1)$ and a constant $C_E \geq0$ such that
			\begin{equation*}
				\| u \|_q \leq C_E \| u \|_{1,2}^{1-\theta} \| u \|_2^{\theta} \qquad (u \in V),
			\end{equation*}
			where $q \in (2,\infty]$ is defined by $\nicefrac{1}{q} = \nicefrac{1}{2}-\nicefrac{(1-\theta)}{d}$.
			
			\item \emph{extension property} if there is a bounded extension operator $\ext \colon V \to \W^{1,2}(\R^d)$, which should satisfy the additional $\L^2$-bound $\|\ext (u) \|_2 \leq C_E \| u \|_2$ for some $C_E > 0$ and all $u \in V$ if we work in dimension $d=2$.
		\end{enumerate}
	\end{definition}
	
	\begin{remark} \label{Traces: Rem: (Ext) implies (V)}
		In Section~\ref{Subsec: Concrete geometry} we come back to the extension property in concrete settings. The extension property implies the embedding property. Indeed, if $d\geq 3$, then we can take $q= \nicefrac{2d}{(d-2)}$ and $\theta =0$, and use the commutative diagram
		\begin{equation}
			\label{eq: commutative diagram}
			\begin{tikzcd}[row sep=large, column sep=huge]
				\W^{1,2}(\R^d)
				\arrow{r}{\subseteq} 
				& 
				\L^{q}(\R^d)
				\arrow{d}{|_O}
				\\
				V 
				\arrow{u}{\ext}
				\arrow{r}{\subseteq}
				& \L^{q}(O)
			\end{tikzcd}
		\end{equation}
		where the first line is the Sobolev embedding on $\R^d$. If $d=2$, then we can take any $\theta \in (0,1)$ and transfer the respective Gagliardo--Nirenberg inequality from $\R^d$ to $O$ in the same manner. This is where the additional $\L^2$-bound for $\ext$ is needed.
	\end{remark}
	
	The embedding property ensures that $O$ is thick enough in points away from $D$. Indeed, the next result is a generalization of \cite[Thm.~1]{HKT-Measure-density}, see also \cite{Bechtel_Extension-in-Wsp}. Unlike the original result, it also applies in the critical case $d=2$ without any connectivity assumption on $O$, compare with the proof of \cite[Lem.~13]{HKT-Measure-density}. 
	
	\begin{proposition} \label{Traces: Prop: Degenerate ITC}
		If $V$ has the embedding property, then there is $C > 0$ depending only on $d, \theta$ and $C_E$ such that
		\begin{equation*}
			C r^d \leq  | O \cap B(x,r) |  
		\end{equation*}
		for all $x \in \overline{O} \setminus D$ and $r \in (0, 1 \land \dist(x,D))$.
	\end{proposition}
	
	\begin{proof}
		We fix $x$ and $r$. As in \cite{HKT-Measure-density}, we consider a `halving radius' $\widehat{r} \in (0,r)$ such that 
		\begin{equation*}
			|O \cap B(x, \widehat{r})| = \frac{1}{2} |O \cap B(x,r)|. 
		\end{equation*}
		We fix $\varphi \in \smooth[\R^d]$ such that $\1_{B(x, \widehat{r})} \leq \varphi \leq \1_{B(x, r)}$ and $\| \nabla \varphi \|_{\infty} \leq \nicefrac{c}{(r - \widehat{r})}$ with a dimensional constant $c \geq 1$. Then $\varphi|_O \in V$ with estimates
		\begin{align*}
			\|\varphi\|_2 &\leq |O \cap B(x,r)|^{\frac 12},\\
			\|\varphi\|_{1,2} &\leq 2c(r-\widehat{r})^{-1}|O \cap B(x,r)|^{\frac 12},\\
			\|\varphi\|_q &\geq \Big(\frac{1}{2} |O \cap B(x,r)|\Big)^{\frac 12 - \frac{(1-\theta)}{d}},
		\end{align*}
		where for the Sobolev norm we have used that $r-\widehat{r} \leq 1$. We plug this bound into the embedding property in order to obtain 
		\begin{equation*}
			r - \widehat{r} \leq 2 c (2^{\nicefrac{1}{q}} C_E)^{\frac{1}{(1 - \theta)}} |O \cap B(x,r)|^{\frac{1}{d}}.
		\end{equation*}
		
		Now, we iterate: $r_1 \coloneqq r$ and $r_{n+1} \coloneqq \widehat{r_n}$. Since $|O \cap B(x, r_n)| = 2^{-n} |O \cap B(x,r)|$ tends to $0$ as $n \to \infty$, we find $r_n \to 0$. A telescoping series yields the claim
		\begin{align*}
			r &= \sum_{n=1}^{\infty} (r_{n} - r_{n+1})
			\\ &\leq 2 c (2^{\nicefrac{1}{q}} C_E)^{\frac{1}{(1 - \theta)}} \sum_{n=1}^{\infty} |O \cap B(x, r_n)|^{\frac{1}{d}} 
			\\&= 2 c (2^{\nicefrac{1}{q}} C_E)^{\frac{1}{(1 - \theta)}} \bigg( \sum_{n=1}^{\infty} 2^{- \frac{n}{d}} \bigg) |O \cap B(x,r)|^{\frac{1}{d}} 
			\\&= \frac{2 c (2^{\nicefrac{1}{q}} C_E)^{\frac{1}{(1 - \theta)}}}{2^{\nicefrac{1}{d}}-1} |O \cap B(x,r)|^{\frac{1}{d}}. \qedhere
		\end{align*}
	\end{proof}
	
	With the previous result at hand, we can prove that our various notions of traces coincide on $\Sigma$. 
	
	\begin{corollary} \label{Traces: Cor: Tr_I = Tr_G for Lebesgue point}
		Suppose that $V$ has the embedding property. Let $f$ be integrable on bounded subsets of $O$ and let $u \in V \cap \rC(\overline{O})$. 
		\begin{enumerate}
			\item \label{Traces: Cor: Tr_I = Tr_G for Lebesgue point: i} If $\ext (f) \in \L^1_\loc(\R^d)$ is any extension of $f$, then at every Lebesgue point $x \in \Sigma$ of $\ext (f)$ we have
			\begin{equation*}
				\tr_{\glob, \Sigma}(\ext (f))(x) = \tr_{\mathrm{int},\Sigma}(f)(x).
			\end{equation*}
			
			\item \label{Traces: Cor: Tr_I = Tr_G for Lebesgue point: ii} For every $x \in \Sigma$ we have
			\begin{equation*}
				\tr(u)(x) = \tr_{\mathrm{int}, \Sigma} (u)(x).
			\end{equation*}
		\end{enumerate}
	\end{corollary}
	
	\begin{proof}[\rm\bf{Proof of (i)}.] We set $z \coloneqq \tr_{\glob, \Sigma}(\ext (f))(x)$. For $r \in (0, 1 \land \dist(x,D))$ we obtain from Proposition~\ref{Traces: Prop: Degenerate ITC} that 
		\begin{equation*}
			\bigg| \barint_{O \cap B(x,r)} f \, \d y - z \bigg|
			\leq \barint_{O \cap B(x,r)} |f(y) -z| \, \d y
			\leq \frac{|B(0,1)|}{C} \barint_{B(x,r)} |\ext (f)(y) - z| \, \d y. 
		\end{equation*}
		The right-hand side converges to $0$ as $r \to 0$ since $x$ is a Lebesgue point of $\ext (f)$. Hence, $\tr_{\mathrm{int}, \Sigma}(f)(x)$ exists and equals $z$.
		
		\textbf{Proof of (ii).} This follows since $u$ is continuous at $x$.	
	\end{proof}
	
	\subsection{Continuity of the trace}    \label{Subsec: Continuity of the trace}
	
	In order to get a continuous trace map into $\L^2(\Sigma)$ as required in Assumption~\ref{Ass: DBC}, we need to guarantee that Sobolev functions have sufficiently many Lebesgue points on $\Sigma$ and that the so-obtained trace is controlled in norm. For this part only, we switch to a more concrete setup inspired by Jonsson--Wallin~\cite{Jonsson-Wallin}.
	
	We work with the Hausdorff measure $\cH^\ell$ of dimension $\ell \in (d-2,d)$ on $\R^d$. Readers can refer to \cite[Chap.~7]{Yeh_Real-Analysis} for background. In particular, sets with finite $\cH^\ell$-measure are Lebesgue null, and if $F \subseteq \R^d$ is a Borel set with $\cH^{\ell}(F) < \infty$, then the restriction of $\cH^{\ell}$ to $F$ is a Radon measure. The restriction on the dimension stems from a fundamental result in potential theory~\cite[Thm.~6.2.1 \& Thm.~5.1.13]{Hedberg}: every $u \in \W^{1,2}(\R^d)$ has Lebesgue points $\cH^{\ell}$-almost everywhere.

 % \begin{lemma}[{\cite[Thm.~6.2.1 \& Thm.~5.1.13]{Hedberg}}] \label{Traces: Lem: Lebesgue points Hl-ae}
 % Let $\ell \in (d-2,d)$. Then 
 % \end{lemma}
	
	\begin{definition}   \label{Traces: Def: Upper s-set}
		Let $\ell \in (d-2, d)$ and $F \sub \R^d$ be a Borel set. We call $F$ an
		\begin{itemize}
			\item  \textbf{upper $\boldsymbol{\ell}$-set} if there is a constant $C > 0$ such that
			\begin{equation*}
				\cH^\ell (F \cap B(x,r)) \leq C r^{\ell} \qquad (x \in F, r \in (0,1]).
			\end{equation*}
			
			\item  \textbf{$\boldsymbol{\ell}$-set} if there are constants $C, c > 0$ such that
			\begin{equation*}
				c r^{\ell} \leq \cH^\ell (F \cap B(x,r)) \leq C r^{\ell} \qquad (x \in F, r \in (0,1]).
			\end{equation*}
			
		\end{itemize}
	\end{definition}
	
	If $\Sigma$ is an upper $\ell$-set, then we take $m=\cH^\ell|_\Sigma$ as the measure on $\Sigma$. Here is the main result of the section:
	
	\begin{theorem} \label{Traces: Thm: All traces coincide and are bounded}
		Assume that $V$ has the extension property (with extension operator $\ext$) and that $\Sigma$ is an upper $\ell$-set for some $\ell \in (d-2,d)$. Then 
		\begin{equation*}
			\tr_{\glob, \Sigma} (\ext (u)) = \tr_{\mathrm{int}, \Sigma} (u) = \tr (u) \qquad (u\in V)
		\end{equation*}
		and all three linear operators are bounded from $V$ into $\L^2(\Sigma)$. Furthermore, if $\Sigma$ is bounded, then they are compact. 
	\end{theorem}
	
	\begin{proof}
		Since $\ext (u)$ has Lebesgue points $\cH^\ell$-a.e.\ on $\R^d$, we conclude from Corollary~\ref{Traces: Cor: Tr_I = Tr_G for Lebesgue point}~\ref{Traces: Cor: Tr_I = Tr_G for Lebesgue point: i} that $\tr_{\glob, \Sigma} (\ext (u)) = \tr_{\mathrm{int}, \Sigma} (u)$ holds $\cH^\ell$-a.e.\ on $\R^d$. Since $\ell > d-2$, we obtain from \cite[Chap.~VI, Thm.~1 \& Rem.~1]{Jonsson-Wallin} that $\tr_{\glob,\Sigma} \colon \W^{1,2}(\R^d) \to \L^2(\Sigma)$ is bounded. Hence, also $\tr_{\glob, \Sigma} \circ \ext \colon V \to \L^2(\Sigma)$ is bounded. If in addition $u \in \rC(\overline{O})$, then by Corollary~\ref{Traces: Cor: Tr_I = Tr_G for Lebesgue point}~\ref{Traces: Cor: Tr_I = Tr_G for Lebesgue point: ii} we have $\tr(u) = \tr_{\mathrm{int}, \Sigma} (u)$ and we have already seen that the interior trace is bounded and everywhere defined on $V$. Thus, $\tr_{\mathrm{int}, \Sigma}$ is the continuous extension of $\tr$ to $V$.
		
		It remains to prove that $\tr_{\glob,\Sigma} \circ \ext$ is compact if $\Sigma$ is bounded. To this end, we fix an open ball $B \supseteq \overline{\Sigma}$ and a function $\eta \in \smooth[B]$ with $\eta = 1$ on $\overline{\Sigma}$, so that $\tr_{\Sigma} \circ \ext = \tr_{\Sigma} \circ (\eta \ext)$. The key point is that \cite[Chap.~VI, Thm.~1]{Jonsson-Wallin} even gives continuity of $\tr_{\glob,\Sigma} \colon \rB^{2,2}_{1-\varepsilon}(\R^d) \to \L^2(\Sigma)$ on Besov spaces for sufficiently small $\varepsilon>0$. Similar to \eqref{eq: commutative diagram}, $\tr_{\Sigma} \circ \ext$ factorizes through the embedding $\W^{1,2}_{\partial B}(B) \sub \rB^{2, 2}_{1- \varepsilon}(\R^d)$, which is compact \cite[Cor.~2.96]{FA_and_NLPDE_book}.
	\end{proof}
	
	\begin{corollary}  \label{Traces: Cor: j compact}
		In the setting of Theorem~\ref{Traces: Thm: All traces coincide and are bounded} suppose that $O$ is bounded. Then $j \colon V \to \IL^2$ is compact. 
	\end{corollary}
	
	\begin{proof}
		We already know that $\tr \colon V \to \L^2(\Sigma)$ is compact. It remains to see that the inclusion $V \subseteq \L^2(O)$ is compact. But this follows by taking an open ball $B \supseteq \overline{O}$ and a function $\eta \in \smooth[B]$ with $\eta = 1$ on $\overline{O}$ and factorizing the inclusion through the same compact embedding as before.
	\end{proof}
	
	\subsection{Concrete geometric setups}  \label{Subsec: Concrete geometry}
	The extension property for $V$ holds if $O$ is a \emph{Lipschitz domain near} $\partial O \setminus D$ or more generally, if $O$ is \emph{locally uniform near $\partial O \setminus D$}, see \cite{Kato_Mixed, Extension_Operator} for details. Theorem~\ref{Traces: Thm: All traces coincide and are bounded} implies the following concrete version of Assumption~\ref{Ass: DBC}. In Figure~\ref{fig. Traces: A geometric constellation} we illustrate a geometric configuration that goes far beyond the Lipschitz class.
	
	\begin{corollary} \label{Traces: Cor: Concrete setup implying Ass}
		Suppose that $D$ is closed, $\Sigma$ is an upper $\ell$-set for some $\ell \in (d-2,d)$ and $O$ is locally uniform near $\partial O \setminus D$. Then Assumption~\ref{Ass: DBC} is satisfied. 
	\end{corollary}

	\begin{figure}
		\centering
		\begin{tikzpicture}[decoration=Koch snowflake]
			
			% Give O a color
			\fill[violet!30] (-2,-2) -- plot [smooth, tension=0.3] coordinates {(-2,0) (-1,0) (0.2,-0.7) (0,0) (1.7, -1) (0.3, 0.8) (2,0)} -- (8,0) -- plot [smooth, tension=0.8] coordinates {(8,0) (7,-0.2) (8.5,-1) (9,-0.6)} -- (9,-2); 
			
			% Draw x- and y-axis
			\draw[->] (-1,0) -- (9,0) node[right] {$x$};
			\draw[->] (0,-2) -- (0,3.3) node[above] {$y$};
			
			% Draw D
			\draw [black, line width = 2] plot [smooth, tension=0.3] coordinates {(2,0) (0.3, 0.8) (1.7, -1) (0,0) (0.2,-0.7) (-1,0) (-2,0)}; 
			\draw[black, line width = 2] (6,0) -- (8,0);
			\draw [black, line width = 2] plot [smooth, tension=0.8] coordinates {(8,0) (7,-0.2) (8.5,-1) (9,-0.6)}; 
			
			% Draw N
			\draw[fill = violet!30, line width = 0.8] decorate{ decorate{ decorate{ decorate{ (2,0) -- (6,0) }}}};
			% Draw Sigma_O
			\draw[line width = 0.8] (3.3,-1) -- (6.7,-1); 
			
			% Decorate with D, N, etc. 
			\node[above] at (6,2) {$O^c$};
			\node[below] at (3,-1.5) {$O$};
			\node[right] at (-1.5,0.3) {$D$};
			\node[right] at (6.5,0.3) {$D$};
			\node[below] at (4.5,-1) {$\Sigma \cap O$}; 
			\node[right] at (4,1.5) {$\Sigma \cap \partial O$};
		\end{tikzpicture}
		\caption{A geometric constellation in $\R^2$ that matches with Corollary~\ref{Traces: Cor: Concrete setup implying Ass} and hence satisfies Assumption~\ref{Ass: DBC}. The set $\Sigma \cap \partial O$ is a part of the von Koch snowflake, which is an upper $\ell$-set for $\ell = \log_3(4) > 1$, see \cite[Sec.~2.3]{Falconer_book}, and we take $m=\cH^\ell$ on this dynamical boundary part. A proof of local uniformity of $O$ near $\partial O \setminus D$ can be found in \cite[Prop.~6.30]{Uniform_domains_book}. The proof of Theorem~\ref{Traces: Thm: All traces coincide and are bounded} shows that $\Sigma$ could be the union of multiple disjoint upper $\ell$-sets with  different values of $\ell$ and we can add a jump condition over the line segment  $\Sigma \cap O$, which is a $1$-set.}.
		\label{fig. Traces: A geometric constellation}
	\end{figure}
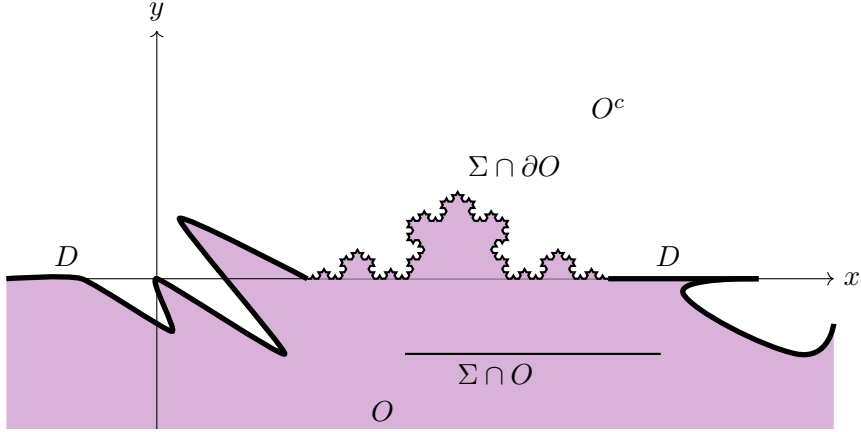
	
	We close this section by extrapolating Theorem~\ref{Traces: Thm: All traces coincide and are bounded} to an admissible range of $\IL^p$-spaces in the geometric setting from above. This will be important for regularity theory, see Section~\ref{Sec: Elliptic regularity for DBC}. 
	
	\begin{theorem} \label{Traces: Thm: Extension to Lp}
		Let $\ell \in (d-2, d)$, assume that $O$ is locally uniform near $\partial O \setminus D$, $\Sigma$ is an upper $\ell$-set and let $\ext$ be the extension operator from \cite{Extension_Operator}. Then the following hold true for all $p \in ((d - \ell) \vee 1, d)$ and $q \in [p, \nicefrac{\ell p}{(d-p)})$: 
		
		\begin{enumerate}
			\item We have 
			\begin{equation*}
				\tr_{\glob, \Sigma} (\ext u) = \tr_{\mathrm{int}, \Sigma} (u) = \tr (u) \qquad (u\in  \W^{1,p}_D(O) \cap V)
			\end{equation*}
			and all three operators admit bounded extensions from $\W_D^{1,p}(O)$ into $\L^q(\Sigma)$.
			
			\item  If $\Sigma$ is bounded, then the operators in (i) are compact.
			
			\item  If $O$ is bounded, then $j$ has a compact extension from $\W^{1,p}_D(O)$ into $\IL^q$. 
		\end{enumerate}
	\end{theorem}
	
	\begin{proof}
		We follow the proof of Theorem~\ref{Traces: Thm: All traces coincide and are bounded} and explain all necessary modifications. 
		
		\textbf{Proof of (i).} Let $s \coloneqq 1 - ( \nicefrac{d}{p} - \nicefrac{d}{q}) \in (0,1]$. Our choice of $p$ and $q$ implies that $1 - \nicefrac{d}{p} = s - \nicefrac{d}{q}$ and $s > \nicefrac{(d- \ell)}{q}$. By \cite[Thm.~1.2]{Extension_Operator}, \cite[Prop.~2.71]{FA_and_NLPDE_book} and \cite[Chap.~VI, Thm.~1 \& Rem.~1]{Jonsson-Wallin}, we have the chain of continuous operators
		\begin{equation*}
			\W^{1,p}_D(O) \overset{\ext}{\longrightarrow} \W^{1,p}(\R^d) \sub \rB^{q,q}_s(\R^d)  \overset{\tr_{\glob, \Sigma}}{\longrightarrow} \L^q(\Sigma).
		\end{equation*}
		Thus, $\tr_{\glob, \Sigma} \circ \ext \colon \W_D^{1,p}(O) \to \L^q(\Sigma)$ is bounded. 
		
		\textbf{Proof of (ii).} We only need to observe that \cite[Chap.~VI, Thm.~1]{Jonsson-Wallin} also gives continuity of $\tr_{\glob,\Sigma} \colon \rB^{q,q}_{s - \varepsilon}(\R^d) \to \L^q(\Sigma)$ for sufficiently small $\varepsilon>0$. Using \cite[Cor.~2.96]{FA_and_NLPDE_book}, we conclude that $\tr_{\Sigma} \circ \ext$ factorizes through the compact embedding $\W^{1,p}_{\partial B}(B) \sub \rB^{q, q}_{s- \varepsilon}(\R^d)$.
		
		\textbf{Proof of (iii).} This is the same argument as in the proof of Corollary~\ref{Traces: Cor: j compact} with $(V, \L^2(O))$ replaced by $(\W^{1,p}_D(O), \L^q(O))$. 
	\end{proof}
	
	%%%%%%%%%%%%%%%%%%%%%%%%%%%%%%%%%%%%%%%%%%%%%%%%%%%%%%%%%%%%%%%%%%%%%%%%%%%%%%%%%%%%%%

	\section{\texorpdfstring{$\L^2$}{}-theory} \label{Sec: L2-theory}
	
	This section contains all relevant definitions and proofs for $p=2$.
	
	\subsection{The extended form method} \label{Subsec: L2 realization of L}
	
	Let $A \colon O \to \C^{d \times d}$ be uniformly strongly elliptic as in \eqref{eq: Intro: ellipticity}. We consider the bounded sesquilinear form
	\begin{equation*}
		a\colon V \times V \to \C, \quad a(u,v) \coloneqq \int_O A \nabla u \cdot \overline{\nabla v} \, \d x, 
	\end{equation*}
	as in \eqref{eq: Intro: a}. Ellipticity implies that $a$ is \emph{$\boldsymbol{j}$-elliptic} in the sense of \cite{Arendt_Elst_j-elliptic_forms}, because 
	\begin{equation*}
		\re a(u,u) + \lambda \|j(u)\|_2^2 \geq \lambda \|u\|_V^2 \qquad (u \in V).
	\end{equation*}
	Moreover, $a$ is \emph{sectorial} and
	\begin{equation*}
		0 \leq \omega(A) \coloneqq  \underset{x \in O}{\esssup} \sup_{|\xi| =1}| \arg(A(x) \xi \cdot \overline{\xi})| \leq \arccos \big(\tfrac{\lambda(A)}{\Lambda(A)} \big) < \tfrac{\pi}{2},
	\end{equation*}
	see also~\cite{Elst-Linke-Rehberg_Numerical_range}.
	The associated operator in $\IL^2$ is defined as follows. We say $u \in \D(\LL^A)$ if and only if there are $\LL^A u \in \IL^2$ and $w \in V$ such that $j(w) = u$ and
	\begin{equation*}
		\langle \LL^A u , j(v) \rangle = a(w,v) \qquad (v \in V).
	\end{equation*}
	Note that $\LL^A u$ is unique since $j$ has dense range by Lemma~\ref{Traces: Lem: j(V) dense}. Abstract theory of sectorial forms~\cite[Thm.~2.1~(ii)]{Arendt_Elst_j-elliptic_forms} implies that $-\LL^A$ generates a strongly continuous semigroup $(\e^{-t \LL^A})_{t \geq 0}$ that extends to an analytic contraction semigroup of angle $\nicefrac{\pi}{2}-\omega(A)$. Hence, $\e^{-z \LL^A}$ is defined for $z \in \S_{\mu}$ with $\mu \in [0, \nicefrac{\pi}{2}-\omega(A))$, where we use the notation
	\begin{equation*}
		\S_{\mu} \coloneqq \{ z \in \C\setminus \{ 0\} \colon |\arg(z)| < \mu \} \quad \& \quad \S_0 \coloneqq (0, \infty)
	\end{equation*}
	for sectors in the complex plane. Let us also remark that $\LL^A$ admits compact resolvents if $j \colon V \to  \IL^2$ is compact~\cite[Lem.~2.7]{Arendt_Elst_j-elliptic_forms}. This happens in the setting of Theorem~\ref{Traces: Thm: All traces coincide and are bounded} when $O$ is bounded, compare with Corollary~\ref{Traces: Cor: j compact}.
	
	\subsection{\texorpdfstring{A formula for \boldmath$\LL^A$}{}} \label{Subsec: Formula for L via Lax-Milgram}
	
	The `Lax--Milgram' operator associated with the form $a$ is defined as 
	\begin{equation*}
		\cL \colon V \to V^*, \quad \langle \cL u \, | \, v \rangle = a(u,v) \qquad (u,v \in V).
	\end{equation*}
	Just as $a$, this operator is independent of the boundary dynamics and it is natural to ask for a formula relating $\LL^A$ and $\cL$. If $\cL$ is invertible, then $j$ and its adjoint provide the link between $\cL^{-1}$ and $(\LL^A)^{-1}$ as in the following lemma. It has nothing to do with the concrete choices of $a$, $V$ and $j$, and is valid in the general $j$-elliptic framework of \cite{Arendt_Elst_j-elliptic_forms}. 
	
	\begin{lemma} \label{Formula for L via L-M: Lem: Repres via j}
		Suppose that $\cL$ is invertible. Then, so is $\LL^A$ and 
		\begin{equation*}
			(\LL^A)^{-1} = j \cL^{-1} j^*.
		\end{equation*}
	\end{lemma}
	
	\begin{proof}
		To prove that $\LL^A$ is injective, we let $u \in \D(\LL^A)$ with $\LL^A u =0$. Then $u = j(w)$ for some $w \in V$ and 
		\begin{equation*}
			0 = a(w,v) = \langle \cL w \, | \, v \rangle \qquad (v \in V). 
		\end{equation*}
		Hence, $\cL w =0$ and as $\cL$ is injective, we conclude that $w = 0$. Consequently, $u = j(w) =0$.
		
		As for surjectivity, we take any $f \in \IL^2$. Since $\cL$ is surjective, we find $w \in V$ such that $\cL w = j^* (f)$, that is 
		\begin{equation*}
			\langle f, j(v) \rangle = \langle j^* (f) \, | \, v \rangle = \langle \cL w \, | \, v \rangle = a(w,v) \qquad (v \in V). 
		\end{equation*}
		This means that $j(w) \in \D(\LL^A)$ with $\LL^A j(w) = f$. In total, $\LL^A$ is invertible with inverse $(\LL^A)^{-1} f = j (w) = j \cL^{-1} j^* f$.  
	\end{proof}
	
	\begin{remark} \label{Formula for L via L-M: Rem: Coercive a} By standard form theory, $\cL$ is invertible if $a$ is \emph{coercive}, that is, if there is $C > 0$ such that 
	\begin{equation*}
		\re a(u,u) \geq C \| u \|_V^2 \qquad (u \in V),
	\end{equation*}
	or, equivalently, if $V$ admits the global Poincaré inequality $\| u \|_2 \leq C \| \nabla u \|_2$ for some $C > 0$ and all $u \in V$. If $O$ is a bounded domain and $V$ has the extension property, then Poincar\'e's inequality always holds unless $V = \W^{1,2}_{\varnothing}(O)$ models good Neumann boundary conditions~\cite[Lem.~6]{Egert_p-ellipticity}.
	\end{remark}
	
	Lemma~\ref{Formula for L via L-M: Lem: Repres via j} allows us to transfer mapping properties from the `non-dynamical' operator $\cL^{-1}$ to the `dynamical operator' $(\LL^A)^{-1}$. We shall give a striking application in Section~\ref{Sec: Elliptic regularity for DBC}. Let us stress that for the functional calculus of $\LL^A$ no such simple transference can be used (and hence the results in this paper are non-trivial). Indeed, already for resolvents we do not have the formula 
	\begin{equation*}
		(t + \LL^A)^{-1} = j (t + \cL)^{-1} j^* \qquad (t > 0), 
	\end{equation*}
	because the form corresponding to the left-hand side is $a(\cdot, \cdot) + t \langle j(\cdot), j(\cdot) \rangle_{\IL^2}$ and not just $a(\cdot, \cdot) + t \langle \cdot, \cdot \rangle_{\L^2(O)}$.

	\subsection{\texorpdfstring{The bilinear embedding for \boldmath$p=2$}{}} \label{Subsec: Bilinear estimate, p=2}
	
	We find it instructive to give an elementary proof of Theorem~\ref{Intro: Thm: BE} in the case $p=2$ first. We fix $f, g \in \IL^2$ and for $t > 0$ we abbreviate
	\begin{equation*}
		f_t^A \coloneqq \e^{- t \LL^A} f  \quad \& \quad g_t^B \coloneqq \e^{-t \LL^B} g.
	\end{equation*}
	We consider the power function
	\begin{equation*}
		Q(\zeta, \eta) \coloneqq |\zeta|^2 + |\eta|^2 \qquad (\zeta, \eta \in \C)
	\end{equation*}
	and define the corresponding heat flow as
	\begin{equation*}
		\energy(t) \coloneqq \int_{O \cup \Sigma} Q(f_t^A, g_t^B ) \, \d \mu = \| f_t^A \|_2^2 + \| g_t^B \|_2^2.
	\end{equation*}
	Since $(\e^{-t \LL^A})_{t \geq 0}$ and $(\e^{-t \LL^B})_{t \geq 0}$ are bounded analytic $\rC_0$-semigroups in $\IL^2$, we obtain $\energy \in \rC[0, \infty) \cap \rC^1(0, \infty)$ and
	\begin{equation*}
		- \energy'(t) = 2 \re \int_{O \cup \Sigma} \LL^A f_t^A \cdot \overline{f_t^A} + \LL^B g_t^B \cdot \overline{g_t^B} \, \d \mu.
	\end{equation*}
	Analyticity of the semigroups also entails that $f_t^A \in \D(\LL^A)$ and $g_t^B \in \D(\LL^B)$. In particular, $f_t^A, g_t^B \in j(V)$ with $f_t^A = j(f_t^A|_O)$ and $g_t^B = j( g_t^B|_O)$. Hence, ellipticity and the elementary inequality $2 XY \leq X^2 + Y^2$ yield the lower bound
	\begin{align*}
		- \energy'(t) &= 2  \re \int_O A \nabla (f_t^A|_O) \cdot \overline{\nabla (f_t^A|_O)} + B \nabla (g_t^B|_O) \cdot \overline{\nabla (g_t^B|_O)} \, \d x
		\\&\geq 2 \lambda(A, B) \int_O \big|\nabla(f_t^A|_O)\big|^2 + \big|\nabla(g_t^B|_O)\big|^2 \, \d x
		\\&\geq 4 \lambda(A, B) \int_O \big| \nabla(f_t^A|_O)\big| \cdot \big|\nabla(g_t^B|_O)\big| \, \d x.
	\end{align*}
	Integration in $t \in [0,T]$ leads us to the estimate
	\begin{equation*}
		4 \lambda(A, B) \int_0^T \int_O \big| \nabla(f_t^A|_O)\big| \cdot \big|\nabla(g_t^B|_O)\big| \, \d x \, \d t \leq \energy(0) - \energy(T) \leq \energy(0).
	\end{equation*}
	Sending $T \to \infty$, we obtain
	\begin{equation*}
		4 \lambda(A, B) \int_0^{\infty} \int_O \big| \nabla(f_t^A|_O)\big| \cdot \big|\nabla(g_t^B|_O)\big| \, \d x \, \d t \leq \energy(0) = \| f \|_2^2 + \| g \|_2^2.
	\end{equation*}
	Finally, we replace the pair $(f,g)$ by $(sf, s^{-1} g)$ for $s > 0$ and optimize in $s$ to deduce the bilinear estimate
	\begin{equation*}
		4 \lambda(A, B) \int_0^{\infty} \int_O \big| \nabla(f_t^A|_O)\big| \cdot \big|\nabla(g_t^B|_O)\big| \, \d x \, \d t \leq 2 \| f \|_2 \| g \|_2.
	\end{equation*}
	The main obstacle in generalizing this argument to $p \neq 2$ lies in finding the correct $p$-adapted version of $Q$ and getting $p$-adapted estimates of the semigroups from above and below. Here, the $p$-ellipticity assumption on $A$ and $B$ plays a key role.
	
	\section{\texorpdfstring{$p$-ellipticity and $\L^p$-contractivity of the semigroup}{}} \label{Sec: Lp-contractivity}
	
	% We start this section by introducing the notion of $p$-ellipticity. Afterwards we use this condition on $A$ and an extension of Nittka's result \cite[Thm.~4.1]{Nittka} to $j$-elliptic forms to deduce that $(\e^{-t \LL^A})_{t \geq 0}$ is $\IL^p$-contractive. Last but not least, we end this section by proving Theorem~\ref{Intro: Theorem: Analytic Lp semigroup}.
	
	\subsection{\texorpdfstring{Facts about \boldmath$p$-ellipticity}{}} \label{Subsec: p-ellipticity}
	
	Let us recall that $A$ is elliptic with parameters $\lambda(A), \Lambda(A)$ as in \eqref{eq: Intro: ellipticity} and that it is $p$-elliptic if \eqref{eq: Intro: p-ellipticity} holds true. The following elementary properties will be useful throughout this work, see \cite[Sec.~1.2 \& Sec.~5.3]{CD_JEMS}.

	\begin{lemma} \label{p-ellipticity: Lem: Basics}
		The following hold true.
		
		\begin{enumerate}
			\item\label{p-ellipticity: Lem: Basics, L2} $\Delta_2(A) = \lambda(A)$.
			
			\item\label{p-ellipticity: Lem: Basics, duality in A} $\Delta_p(A^*) \geq \Delta_p(A) (\nicefrac{p}{p'} \land \nicefrac{p'}{p})$.
			
			\item\label{p-ellipticity: Lem: Basics, duality in p}  $\Delta_p(A) = \Delta_{p'}(A)$.
			
			\item\label{p-ellipticity: Lem: Basics, A real} $A$ is real-valued if and only if $\Delta_p(A) > 0$ for all $p \in (1, \infty)$.
			
			\item\label{p-ellipticity: Lem: Basics, Delta Lip}  The map $p \mapsto \Delta_p(A)$ is decreasing and Lipschitz continuous on $[2, \infty)$.
			
			\item\label{p-ellipticity: Lem: Basics, Delta cont in A}  The map $\theta \mapsto \Delta_p(\e^{\i \theta} A)$ is continuous on $(- \nicefrac{\pi}{2}, \nicefrac{\pi}{2})$.
		\end{enumerate}
	\end{lemma}
	
	\subsection{Nittka's invariance criterion \texorpdfstring{for $\boldsymbol{j}$-elliptic forms}{}}
	
	We denote by $P_q$ the projection from $\IL^2$ onto
	\begin{equation*}
		B_q \coloneqq \{ u \in \IL^q \cap \IL^2 \colon \Vert u \Vert_q \leq 1 \}.
	\end{equation*}
	It is convex and closed by Fatou's lemma. The following invariance criterion of closed and convex sets is due to Ouhabaz \cite{Ouhabaz_Invariance}. It continuous to hold in the $j$-elliptic setting \cite[Prop.~2.9~(i),(ii)]{Arendt_Elst_j-elliptic_forms}, see also \cite[Rem.~4.2]{Nittka}. For our injective $j$, it takes the following, simpler form.
	
	\begin{proposition} \label{Lp contractivity: Prop: Ouhabaz j-elliptic}
		Let $q \in [1, \infty]$. The following assertions are equivalent.
		
		\begin{enumerate}
			\item  The semigroup $(\e^{-t \LL^A})_{t \geq 0}$ is $\IL^q$-contractive, that is
			\begin{equation*}
				\| \e^{- t \LL^A} f \|_q \leq \| f \|_q \qquad (t \geq 0, f \in \IL^q \cap \IL^2).
			\end{equation*}
			\item The set $j(V)$ is invariant under $P_q$ and $\re a (j^{-1}(P_q u), j^{-1}(u - P_q u)) \geq 0$ for all $u \in j(V)$.
		\end{enumerate}
	\end{proposition}
	
	First, we use Proposition~\ref{Lp contractivity: Prop: Ouhabaz j-elliptic} to get the invariance of $j(V)$ under $P_q$ for all $q \in [1, \infty]$ by complex interpolation and the fact that $(\e^{-t \LL^A})_{t \geq 0}$ is $\IL^{\infty}$-contractive when $A = I$ is the identity. To this end, we need the precise form of $j$. 
	
	\begin{lemma} \label{Lp contractivity: Lem: L_infinity-contractive}
		The semigroup $(\e^{-t \LL^I})_{t \geq 0}$ is $\IL^{\infty}$-contractive. In particular, $j(V)$ is invariant under $P_q$ for all $q \in [1, \infty]$.
	\end{lemma}
	
	\begin{proof}
		As explained above, we only need the $\L^{\infty}$-contractivity. The projection from $\IL^2$ onto $B_{\infty}$ is given by 
		\begin{equation*}
			P_{\infty} u = (|u| \land 1) \sgn(u) \eqqcolon \Phi(u),   
		\end{equation*}
		where $\sgn(z) \coloneqq \nicefrac{z}{|z|} \1_{\{|z|\neq 0\}}$ and $\Phi(z) \coloneqq (|z| \land 1) \sgn(z)$ is Lipschitz continuous with $\Phi(0) =0$. Hence, Lemma~\ref{Traces: Lem: j commutes!} asserts that $j(V)$ is invariant under $P_{\infty}$. Let $u \in j(V)$. Since $j^{-1}$ is the pointwise restriction to $O$, we deduce 
		\begin{equation*}
			j^{-1}(P_{\infty} u) = \Phi(u|_O) \quad \& \quad j^{-1}(u - P_{\infty} u) = u|_O - \Phi( u|_O). 
		\end{equation*}
		Let $w \coloneqq u|_O$. Since $A = I$, we obtain 
		\begin{equation*}
			\re a (j^{-1}(P_{\infty} u), j^{-1}(u - P_{\infty} u)) = \re \int_O \nabla \Phi(w) \cdot \overline{(\nabla w - \nabla \Phi(w))} \, \d x.
		\end{equation*}
		By \cite[Prop.~4.11]{Ouhabaz_book}, the weak gradient of $\Phi(w)$ is given by 
		\begin{align*}
			\nabla \Phi(w) &= \i \frac{\im(\sgn(\overline{w}) \nabla (w))}{|w|} \sgn(w) \1_{\{ |w| > 1 \}} + \1_{\{ |w| \leq 1 \}} \nabla w
			\\&\eqqcolon \i \Psi(w) \1_{\{ |w| > 1 \}} + \1_{\{ |w| \leq 1 \}} \nabla w.
		\end{align*}
		Inserting this identity, we obtain 	
		\begin{align*}
			\re a (j^{-1}(P_{\infty} u), j^{-1}(u - P_{\infty} u)) &= \re \int_{\{ |w| > 1 \}} \i \Psi(w) \cdot \overline{(\nabla w - \i \Psi(w))} \, \d x
			\\&= \int_{\{ |w| > 1 \}} \frac{|\im(\sgn(\overline{w}) \nabla w)|^2}{|w|} - |\Psi(w)|^2 \, \d x 
            \\&= \int_{\{ |w| > 1 \}} |\Psi(w)|^2 |w| - |\Psi(w)|^2 \, \d x 
			\geq 0. 
		\end{align*} 
		Proposition~\ref{Lp contractivity: Prop: Ouhabaz j-elliptic} yields the claim. 
	\end{proof}
	
	One of Nittka's contributions in \cite{Nittka} is an equivalent formulation of Proposition~\ref{Lp contractivity: Prop: Ouhabaz j-elliptic} (ii) that is more practical in applications. He observed the following fact. 

    \begin{lemma}[{\cite[Prop.~2.4 \& Lem.~3.1]{Nittka}}] \label{Lp contractivity: Lem: Nittka P_q rep}
        Let $q \in [2, \infty)$ and $f \in \IL^2$. There are unique $u \in B_q$ and $t \geq 0$ such that $f = u + t |u|^{q-2} u$. Moreover, $u = P_q f$. 
    \end{lemma}
    
    Armed with this result, we obtain the following version of Nittka's invariance criterion \cite[Thm.~4.1]{Nittka} for our specific setting of $V$ and $j$. 
	
	\begin{proposition}   \label{Lp contractivity: Prop: Nittka}
		Let $q \in [2, \infty)$. The following assertions are equivalent.
		
		\begin{enumerate}
			\item The semigroup $(\e^{- t \LL^A})_{t \geq 0}$ is $\IL^q$-contractive.
			
			\item We have $\re a (u|_O, (|u|^{q-2} u)|_O) \geq 0$ for all $u \in j(V)$ with $|u|^{q-2} u \in j(V)$.
		\end{enumerate}
	\end{proposition}

	\begin{proof}
		First, we assume that (ii) is valid. To show part (ii) of Proposition~\ref{Lp contractivity: Prop: Ouhabaz j-elliptic}, we fix $f \in j(V)$. Lemma~\ref{Lp contractivity: Lem: L_infinity-contractive} yields $u \coloneqq P_q f \in j(V)$. By Lemma~\ref{Lp contractivity: Lem: Nittka P_q rep} we find $t \geq 0$ such that $f = u + t |u|^{q-2} u$. Hence, $f - u = t |u|^{q-2} u \in j(V)$ and
		\begin{equation*}
			\re a(j^{-1}(P_q f), j^{-1}(f - P_q f)) = t \re a(u|_O, (|u|^{q-2} u)|_O) \geq 0.
		\end{equation*}
		Now, we assume that Proposition~\ref{Lp contractivity: Prop: Ouhabaz j-elliptic} (ii) holds true. Let $u \in j(V)$ with $|u|^{q-2} u \in j(V)$. We note that $u \in \IL^q$. If $u =0$ there is nothing to prove. So we assume that $u \neq 0$ and put $\alpha \coloneqq \| u \|_q^{-1} >0$. Then $f \coloneqq \alpha u + |u|^{q-2} u \in j(V)$ and Lemma~\ref{Lp contractivity: Lem: Nittka P_q rep} implies that $P_q f = \alpha u$. Consequently,
		\begin{equation*}
			\re a (u|_O, (|u|^{q-2} u)|_O) = \alpha^{-1} \re a (j^{-1}(P_q f), j^{-1}(f - P_q f)) \geq 0. \qedhere
		\end{equation*}
	\end{proof}
	
	\begin{corollary}  \label{Lp contractivity: Cor: A p-elliptic gives Lp contractive}
		Let $p \in (1, \infty)$ and $A$ be $p$-elliptic. Then $(\e^{-t \LL^A})_{t \geq 0}$ is $\IL^p$-contractive.
	\end{corollary}
	
	\begin{proof}
		By duality and items~\ref{p-ellipticity: Lem: Basics, duality in A} and~\ref{p-ellipticity: Lem: Basics, duality in p} of Lemma~\ref{p-ellipticity: Lem: Basics} we can assume that $p \geq 2$. In view of Proposition~\ref{Lp contractivity: Prop: Nittka} we have to show for all $u \in j(V)$ with $|u|^{p-2} u \in j(V)$ that
		\begin{equation*}
			\re a(w, |w|^{p-2} w) \geq 0, 
		\end{equation*}
        where $w \coloneqq u|_O$. Since $w, |w|^{p-2} w \in V$, this is the crucial estimate that follows from $p$-ellipticity of $A$, see for instance \cite[Cor.~12]{Egert_p-ellipticity}.
	\end{proof}
	
	% \begin{remark}
	% 	We remark that Nittka's result can be generalized to all $j$-elliptic forms that are associated with an m-accretive operator on $\L^2(\Omega, \mu)$, where $(\Omega, \mu)$ is an arbitrary measure space. However, the verification of the lower bound in part (ii) of Proposition~\ref{Lp contractivity: Prop: Ouhabaz j-elliptic} relies heavily on our specific choice of $j$.
	% \end{remark}
	
	\subsection{\texorpdfstring{Generators in \boldmath$\IL^p$}{}}
	
	Let us prove the first assertion of Theorem~\ref{Intro: Thm: Bounded FC in Lp}, namely that $(\e^{-t \LL^A})_{t \geq 0}$ extrapolates to an analytic $\rC_0$-semigroup of contractions in $\IL^p$. 
	
	\begin{proof}[\rm\bf{Proof of Theorem~\ref{Intro: Thm: Bounded FC in Lp}, part 1.}]
		Lemma~\ref{p-ellipticity: Lem: Basics}~\ref{p-ellipticity: Lem: Basics, Delta cont in A} yields $\theta_p > 0$ and part~\ref{p-ellipticity: Lem: Basics, Delta Lip} of the same lemma implies $\nicefrac{\pi}{2} - \theta_p \geq \nicefrac{\pi}{2} - \theta_2 \geq \omega(A)$. Let $\theta \in (- \theta_p, \theta_p)$. Since $\e^{-t \e^{\i \theta} \LL^A} = \e^{- t \LL^{\e^{\i \theta} \!\! A}}$ for all $t \geq 0$, Corollary~\ref{Lp contractivity: Cor: A p-elliptic gives Lp contractive} entails 
		\begin{equation*}
			\| \e^{- z \LL^A} f \|_p \leq \| f \|_p \qquad ( z \in \S_{\theta_p} \cup \{ 0 \}, f \in \IL^p \cap \IL^2).
		\end{equation*}
		By density, each $\e^{-z \LL^A}$ extends to a contractive linear operator $T_p(z)$ on $\IL^p$. In view of Lemma~\ref{p-ellipticity: Lem: Basics}~\ref{p-ellipticity: Lem: Basics, Delta Lip}, the interval of exponents $p$, in which this conclusion holds, is open. Hence, we conclude by \cite[Prop.~3.12]{Ouhabaz_book} that $T_p \coloneqq (T_p(z))_{z \in \S_{\theta_p} \cup \{ 0 \}}$ defines an analytic $\rC_0$-semigroup of contractions in $\IL^p$.
	\end{proof}
	
	The above result gives rise to a proper notion of an $\IL^p$-realization of $\LL^A$.
	
	\begin{definition}
		Let $- \LL_p^A$ be the generator of the semigroup $T_p$. We call $\LL_p^A$ the 
		\textbf{$\boldsymbol{\IL^p}$-realization} of $\LL^A$.
	\end{definition}
	
	Letting $\LL^A|_{\IL^p}$ be the part of $\LL^A$ in $\IL^p$, that is 
	\begin{equation*}
		\LL^A|_{\IL^p} \coloneqq \LL^A \cap (\IL^p \times \IL^p),
	\end{equation*}
	we have the following, more concrete description for the $\IL^p$-realization. 
	
	\begin{lemma}[{\cite[Lem.~3.1]{Do_p-ellipticity}}]
            \label{Lp contractivity: Lem: part in Lp}
		Let $p \in (1, \infty)$ and $A$ be $p$-elliptic. The operator $\LL_p^A$ is the closure of $\LL^A|_{\IL^p}$ in $\IL^p$. In particular, if $p\in (2,\infty)$ and $\mu$ is bounded, then $\LL_p^A = \LL^A|_{\IL^p}$.
	\end{lemma}
	
	Clearly, $\LL_p^A$ is sectorial of angle at most $\nicefrac{\pi}{2} - \theta_p$ in $\IL^p$. For convenience of the reader, we collect elementary properties of $\LL_p^A$.
	
	\begin{lemma} \label{Lp contractivity: Lem: Lp realization}
		Let $p \in (1, \infty)$ and $A$ be $p$-elliptic. 
		
		\begin{enumerate}
			\item \label{Lp contractivity: Lem: Lp realization, Resolvents consistent}
			We have $(z- \LL_p^A)^{-1} f = (z - \LL^A)^{-1} f$ for all $z \in \C \setminus \overline{\S_{\nicefrac{\pi}{2} - \theta_p}}$ and $f \in \IL^p \cap \IL^2$.
			
			\item \label{Lp contractivity: Lem: Lp realization, L_p consistent}
			We have $\LL_p^A u = \LL^A u$ for all $u \in  \D(\LL_p^A) \cap \D(\LL^A)$.
			
			\item \label{Lp contractivity: Lem: Lp realization, Injective}
			If $\LL^A$ is injective, then so is $\LL_p^A$.
			
			\item \label{Lp contractivity: Lem: Lp realization, L invertible}
			If $a$ is coercive, then $\LL_p^A$ is invertible and (i) holds for $z =0$. 
		\end{enumerate}
		
	\end{lemma}
	
	\begin{proof} 
		Parts \ref{Lp contractivity: Lem: Lp realization, Resolvents consistent} and \ref{Lp contractivity: Lem: Lp realization, L_p consistent} are already in \cite[Lem.~3.1]{Do_p-ellipticity}. Let us prove \ref{Lp contractivity: Lem: Lp realization, Injective}. For any $f \in \IL^p \cap \IL^2$ we have $\lim_{t \to \infty} T_p(t) f = \lim_{t \to \infty} T_2(t) f = 0$ strongly in $\IL^2$. Since $T_p$ is uniformly bounded on $\IL^p$ and $\IL^p \cap \IL^2$ is dense in $\IL^p$, the same convergence holds weakly in $\IL^p$ for every $f \in \IL^p$. Now, if $f \in \Ker(\LL_p^A)$, then $T_p(t) f = f$ for all $t > 0$ and $f = 0$ follows.    
		
		To prove \ref{Lp contractivity: Lem: Lp realization, L invertible}, we note that the coercivity of $a$ implies that $T_2$ is exponentially stable. Since $A$ is $p$-elliptic, we can combine Lemma~\ref{p-ellipticity: Lem: Basics}~\ref{p-ellipticity: Lem: Basics, Delta Lip} and part 1 of Theorem~\ref{Intro: Thm: Bounded FC in Lp} to find $q \in (1,\infty)$ such that $p$ lies between $2$ and $q$, and $T_q$ is contractive. By complex interpolation, $T_p$ is exponentially stable. Hence, $\LL_p^A$ is invertible and consistency of the inverse on $\IL^p \cap \IL^2$ follows from consistency of the semigroups by taking the Laplace transform at $\lambda =0$.
	\end{proof}
	
	\section{Bilinear embedding} \label{Sec: BE}
	
	In this section, we prove the bilinear embedding, Theorem~\ref{Intro: Thm: BE}, for $p \neq 2$. By symmetry of the assumptions we suppose from now on that $p > 2$ and that $A$ is $p$-elliptic. We fix $t > 0$, $f, g \in \IL^p \cap \IL^{p'}$ and let $Q \colon \C \times \C \to [0, \infty)$ be a continuously (real) differentiable function. We are interested in the monotonicity behaviour of the adapted heat flow
	\begin{equation*}
		\energy(t) \coloneqq \int_{O \cup \Sigma} Q(f_t^A, g_t^B) \, \d \mu.
	\end{equation*}
	The fundamental theorem of calculus yields 
	\begin{equation*}
		- \int_0^{\infty} \energy'(t) \, \d t \leq \energy(0) = \int_{O \cup \Sigma} Q(f,g) \, \d \mu, 
	\end{equation*}
	and as in Section~\ref{Subsec: Bilinear estimate, p=2} we need an upper bound for $Q$ and a uniform lower bound for $- \energy'(t)$. The difficulty is to find the correct $Q$. As in \cite{CD_JEMS, CD_Open_Set}, we use the ingenious Nazarov--Treil Bellmann function and we refer to these references for historical background.
	
	\begin{definition} \label{BE Def: Bellmann function}
		Let $\delta \in (0,1]$ and $\zeta, \eta \in \C$. The \textbf{Nazarov--Treil Bellmann function} is defined as
		\begin{equation} \label{eq: BE: Bellman function}
			Q(\zeta, \eta) \coloneqq |\zeta|^p + |\eta|^{p'} + \delta \begin{cases}
				|\zeta|^2 |\eta|^{2 - p'}, \quad &(|\zeta|^p \leq |\eta|^{p'}),  \\\ \frac{2}{p} |\zeta|^p + (1 - \frac{2}{p}) |\eta|^{p'}, \quad &( |\zeta|^p \geq |\eta|^{p'}).
			\end{cases}
		\end{equation}
	\end{definition}
	Here, $\delta$ is a degree of freedom that will be chosen small at a later point. We note that 
	\begin{equation} \label{eq: BE: Upper bound for Q}
		Q(\zeta, \eta) \leq 2( |\zeta|^p + |\eta|^{p'}) \qquad (\zeta, \eta \in \C),
	\end{equation}
	and for $p = 2$ we recover (up to a multiplicative constant) the energy from Section~\ref{Subsec: Bilinear estimate, p=2}. We write 
	\begin{equation*}
		\partial_z = \frac{1}{2} (\partial_x - \i \partial_y)
	\end{equation*}
	for the (Wirtinger) derivative with respect to the complex variable $z = x + \i y$. Since for any $q \in (1, \infty)$ the power function $z \mapsto |z|^q$ is continuously differentiable with derivative 
	\begin{equation*}
		\partial_z |z|^q = \partial_z (z \cdot \overline{z})^{\frac{q}{2}} = \frac{q}{2} (z \cdot \overline{z})^{\frac{q}{2}-1} \partial_z (z \cdot \overline{z}) = \frac{q}{2} |z|^{q-2} \overline{z},
	\end{equation*}
	the very definition of $Q$ implies that $Q \in \rC^1(\C^2)$. As a matter of fact, we need to re-tell some of the results found by Carbonaro--Dragičević in \cite{CD_JEMS, CD_Open_Set}. To simplify the exposition, we shall follow their outline in \cite{CD_Schrodinger}. 
 
    The first part of Theorem~\ref{Intro: Thm: Bounded FC in Lp} implies that $(\e^{-t \LL^A})_{t \geq 0}$ and $(\e^{-t \LL^B})_{t \geq 0}$ are bounded analytic semigroups in $\IL^p$ and $\IL^{p'}$, respectively. Hence, $\energy \in \rC[0, \infty) \cap \rC^1(0, \infty)$ with derivative
	\begin{align}
		\begin{split} \label{eq: BE, E'(t)}
			- \energy'(t) &= 2 \re \int_{O \cup \Sigma} \LL^A f_t^A \cdot (\partial_{\zeta} Q) (f_t^A, g_t^B)  + \LL^B g_t^B \cdot (\partial_{\eta} Q)(f_t^A, g_t^B) \, \d \mu
			\\&= 2 \re \int_{O \cup \Sigma} \LL^A j(f_t^A|_O) \cdot (\partial_{\zeta} Q) (j(f_t^A|_O), j(g_t^B|_O))  
			\\&\qquad \qquad+ \LL^B j(g_t^B|_O) \cdot (\partial_{\eta} Q)(j(f_t^A|_O), j(g_t^B|_O)) \, \d \mu,
		\end{split}
	\end{align}
    see \cite[App.~C, Prop.~C1]{CD_Open_Set}.
	We would like to use the definition of $\LL^A$ and $\LL^B$ to get rid of the hybrid measure space, but this is not immediately possible, because $(\partial_\zeta Q, \partial_\eta Q)$ lacks in regularity and, among other things, we cannot use Lemma~\ref{Traces: Lem: j commutes!} to commute $(\partial_\zeta Q, \partial_\eta Q)$ with $j$. A similar issue has also occurred in \cite{CD_Open_Set} when transforming the heat flow method from $\R^d$ \cite{CD_JEMS} to arbitrary open sets. It has been resolved by approximating $Q$ with smooth functions $R_{n, \nu}$ and we will see that the approximations also fit into the setting of Lemma~\ref{Traces: Lem: j commutes!}. To this end, we need:
	
	\begin{itemize}
		\item a radial function $\varphi \in \smooth[\R^4]$ such that $0 \leq \varphi \leq \1_{B(0,1)}$ and $\int_{\R^4} \varphi \, \d x =1$,
		
		\item a non-negative and radial function $\psi \in \smooth[\R^4]$ with $\psi =1$ on $B(0,3)$ and $\psi =0$ on $B(0,4)^c$.
	\end{itemize}
	
	For $n \in \N$, $\nu \in (0,1]$ and $\omega \in \R^4$ we set
	\begin{equation*}
		\varphi_{\nu} (\omega) \coloneqq \nu^{-4}  \varphi(\nicefrac{\omega}{\nu}) \quad \& \quad \psi_n(\omega) \coloneqq  \psi(\nicefrac{\omega}{n}).
	\end{equation*}
	We identify $\C^2$ explicitly with $\R^4$ via
	\begin{equation*}
		\mathcal{W}_2 \colon \C^2 \to \R^4, \quad \mathcal{W}_2 (z_1, z_2) \coloneqq (\re(z_1), \im(z_1), \re(z_2), \im(z_2))
	\end{equation*}
	and define the convolution of $Q$ with $\varphi_{\nu}$ at $z \in \C^2$ by
	\begin{equation*}
		(Q * \varphi_{\nu} )(z) \coloneqq \int_{\R^4}   Q(z - \mathcal{W}_2^{-1}(\omega)) \varphi_{\nu}(\omega) \, \d \omega.
	\end{equation*}
	
	\begin{definition} \label{BE: Def: Approximation of Q}
		Let $n \in \N$ and $\nu \in (0, 1]$. We define $R_{n, \nu} \coloneqq \psi_n ( Q * \varphi_{\nu}) + P_{n, \nu}$.
	\end{definition}
	
	Here, $P_{n, \nu}$ is a cleverly chosen correction term that makes the generalized Hessian of $R_{n, \nu}$ $(A, B)$-convex, see Definition~\ref{BE: Def: Generalized Hessian} below. Its precise form is not needed and we refer to \cite[Sec.~5.2]{CD_Open_Set}. Let us collect useful properties of $R_{n, \nu}$. We write $D = (\partial_{\zeta}, \partial_{\eta})$ and denote by $D^2 R_{n, \nu}$ the Hessian of $R_{n, \nu}$ in the variables $\zeta$ and $\eta$.
	
	\begin{lemma}[{\cite[Lem.~14 \& Thm.~16]{CD_Open_Set}}] \label{BE: Lem: Properties of approximations}
		Let $n \in \N$ and $\nu \in (0,1]$. Then $R_{n, \nu} \in \rC^{\infty}(\C^2)$ has the following additional properties:
		
		\begin{enumerate}[label=(Q\arabic*)]
			\item \label{BE: Lem: Properties of approximations, D^2 R bounded} We have $D^2 R_{n, \nu} \in \L^{\infty}(\C^2; \C^{2 \times 2})$.
			
			\item \label{BE: Lem: Properties of approximations, (DR)(0)=0} We have $(D R_{n, \nu})(0) =0$.
			
			\item \label{BE: Lem: Properties of approximations,  Pw limits for derivatives} We have
			\begin{align*}
				\lim_{n \to \infty} ((D R_{n, \nu})(z), (D^2 R_{n, \nu})(z)) 
                    &= ((D (Q * \varphi_{\nu})) (z),( D^2 (Q * \varphi_{\nu}))(z) ) \qquad (z \in \C^2) \\
				\qquad \lim_{\nu \to 0} \lim_{n \to \infty} ((D R_{n, \nu})(z), (D^2 R_{n, \nu})(z)) 
                    &= ((D Q) (z),( D^2 Q)(z) ) \qquad (\text{a.e. } z \in \C^2).
			\end{align*}
			
			\item \label{BE: Lem: Properties of approximations, Pw bounds for derivatives} We have $|(D R_{n, \nu})(z)| + |D(Q * \varphi_{\nu})(z)| \leq C( |z|^{p-1} + |z|^{p'-1} )$ for some $C = C(\delta, p) >0$ and all $z \in \C^2$.
		\end{enumerate}
	\end{lemma}

    Let us continue in \eqref{eq: BE, E'(t)}. Using~\ref{BE: Lem: Properties of approximations,  Pw limits for derivatives},~\ref{BE: Lem: Properties of approximations, Pw bounds for derivatives} and the boundedness of the semigroup on $\IL^p$ and $\IL^{p'}$ (part 1 of Theorem~\ref{Intro: Thm: Bounded FC in Lp}), we can apply dominated convergence twice and get 
	\begin{align*}
		- \frac{1}{2} \energy'(t) &= \lim_{\nu \to 0} \lim_{n \to \infty} \re \bigg(\int_{O \cup \Sigma} \LL^A j(f_t^A|_O) \cdot (\partial_{\zeta} R_{n,\nu}) (j(f_t^A|_O), j(g_t^B|_O))  
			\\&\qquad \qquad \qquad \qquad+ \LL^B j(g_t^B|_O) \cdot (\partial_{\eta} R_{n,\nu})(j(f_t^A|_O), j(g_t^B|_O)) \, \d \mu \bigg).
       \end{align*}
       Owing to~\ref{BE: Lem: Properties of approximations, D^2 R bounded} and~\ref{BE: Lem: Properties of approximations, (DR)(0)=0}, $DR_{n, \nu}$ falls into the scope of Lemma~\ref{Traces: Lem: j commutes!}. Thus, we can commute $j$ as desired and obtain
    \begin{align}
    \label{eq: BE, E'(t) part 2}
    \begin{split}
		- \frac{1}{2} \energy'(t) &= \lim_{\nu \to 0} \lim_{n \to \infty} \int_{O} \re \Big ( A \nabla (f_t^A|_O) \cdot \nabla (\partial_{\zeta} R_{n, \nu}) (f_t^A|_O, g_t^B|_O) \Big)
		\\&\qquad \qquad \qquad \qquad + \re \Big( B \nabla (g_t^B|_O) \cdot \nabla (\partial_{\eta} R_{n, \nu}) (f_t^A|_O, g_t^B|_O) \Big) \, \d x.
    \end{split}
    \end{align}
	At this point, $j$ has disappeared and we are exactly in the setting of \cite{CD_Open_Set}. 
 
    Let us focus on the first term in the integral. By the chain rule, applying $\nabla$ to $\partial_{\zeta} R_{n, \nu}$ at the point $(z_1, z_2) \coloneqq (f_t^A|_O, g_t^B|_O) \in \C^2$ produces second-order derivatives of $R_{n, \nu}$ that are multiplied with the gradient pair $(X_1, X_2) \coloneqq (\nabla (f_t^A|_O), \nabla (g_t^B|_O)) \in (\C^d)^2$ and we obtain $(\partial_{\zeta}^2 R_{n, \nu})(z_1, z_2) X_1 + (\partial_{\eta} \partial_{\zeta} R_{n, \nu})(z_1, z_2) X_2$. Thus, the first term in the integral is 
	\begin{equation*}
		\re \Big( A X_1 \cdot (\partial_{\zeta}^2 R_{n, \nu})(z_1, z_2) X_1 + A X_1 \cdot (\partial_{\eta} \partial_{\zeta} R_{n, \nu})(z_1, z_2) X_2 \Big)
	\end{equation*}
    and we have a similar formula for the second one. Their sum is the following object for $F=R_{n,\nu}$.
    
    \begin{definition} \label{BE: Def: Generalized Hessian}
    Let $F \colon \C^2 \to \R$ be twice (real) differentiable, $x \in O$, $z_1, z_2 \in \C$ and $X_1, X_2 \in \C^d$. We define the \emph{generalized Hessian} of $F$ with respect to $A$ and $B$ as 
	\begin{equation} \label{eq: Bilinear estimate: Generalized Hessian}
		H_F^{(A(x), B(x))}[(z_1, z_2); (X_1, X_2)] \coloneqq \re \left( \begin{bmatrix}
			A(x) X_1 \\ B(x) X_2
		\end{bmatrix} \cdot (D^2 F)(z_1, z_2) \begin{bmatrix}
			X_1 \\ X_2
		\end{bmatrix} \right),
	\end{equation}
    where we identify each entry of the Hessian $(D^2 F)(z_1, z_2) \in \C^{2 \times 2}$ with a multiplication operator in $\C^d$.
    \end{definition} 
    
    In this terminology, \eqref{eq: BE, E'(t) part 2} becomes
    \begin{align}
    \label{eq: BE, E'(t) part 3}
		- \frac{1}{2} \energy'(t) &= \lim_{\nu \to 0} \lim_{n \to \infty} \int_{O} H_{R_{n,\nu}}^{(A, B)}\Big[(f_t^A|_O, g_t^B|_O); (\nabla (f_t^A|_O), \nabla (g_t^B|_O))\Big] \, \d x.
    \end{align}

    The beautiful insight of Carbonaro--Dragičević in \cite{CD_JEMS} was to establish the link between $p$-ellipticity of $A$ and $B$ and pointwise lower bounds for certain generalized Hessians, which they call \emph{$\boldsymbol{(A,B)}$-convexity}. 
	
	\begin{proposition}[{\cite[Cor.~5.5]{CD_JEMS} \& \cite[Thm.~16]{CD_Open_Set}}] \label{Bilinear estimate: Prop: (A,B)-convexity}
		Let $\nu \in (0,1]$. There is some $\delta \in (0,1)$, depending only on $\lambda(A,B)$, $\Lambda(A,B)$ and $\Delta_p(A,B)$, such that for almost every $x \in O$, all $z_1, z_2 \in \C$ and every $X_1, X_2 \in \C^d$ we have:
        \begin{enumerate}
        \item The generalized Hessian of $R_{n,\nu}$ is non-negative.
        \item The generalized Hessian of $Q * \varphi_{\nu}$ satisfies the explicit lower bound
		\begin{equation*}
			H_{Q * \varphi_{\nu}}^{(A(x), B(x))}[(z_1, z_2); (X_1, X_2)] \geq \frac{\Delta_p(A, B)}{5} \cdot \frac{\lambda(A, B)}{\Lambda(A, B)} |X_1| |X_2|.
		\end{equation*}
        \end{enumerate}
	\end{proposition}
    
	Since the generalized Hessian of $R_{n,\nu}$ is non-negative almost everywhere, we can use \ref{BE: Lem: Properties of approximations,  Pw limits for derivatives} and Fatou's lemma in \eqref{eq: BE, E'(t) part 3} in order to obtain
	\begin{align*} 
		- \frac{1}{2} \energy'(t) &\geq \liminf_{\nu \to 0} \int_O H_{Q * \varphi_{\nu}}^{(A, B)} \Big[(f_t^A|_O, g_t^B|_O); (\nabla (f_t^A|_O), \nabla (g_t^B|_O) ) \Big] \, \d x
		\\&\geq \frac{\Delta_p(A, B)}{5} \cdot \frac{\lambda(A, B)}{\Lambda(A, B)} \int_O |\nabla (f_t^A|_O)| |\nabla(g_t^B|_O)| \, \d x.
	\end{align*}
	Now, integration with respect to $t \in [0, T]$ gives
	\begin{align*}
		\frac{\Delta_p(A, B)}{5} \cdot \frac{\lambda(A, B)}{\Lambda(A, B)} \int_0^T \int_O \big |\nabla (f_t^A|_O)\big| \cdot \big|\nabla (g_t^B|_O)\big| \, \d x \, \d t &\leq \energy(0) - \energy(T) \leq \energy(0) 
	\end{align*}
	and from \eqref{eq: BE: Upper bound for Q} and letting $T \to \infty$ we obtain
	\begin{equation*}
		\frac{\Delta_p(A, B)}{5} \cdot \frac{\lambda(A, B)}{\Lambda(A, B)} \int_0^{\infty} \int_O \big|\nabla (f_t^A|_O)\big| \cdot \big|\nabla (g_t^B|_O)\big| \, \d x \, \d t \leq 2 \big( \| f \|_p^p + \| g \|_{p'}^{p'} \big).
	\end{equation*}
	Finally, we replace $(f,g)$ by $(s f, s^{-1} g)$ for $s >0$ and optimize in $s$ to conclude the proof of Theorem~\ref{Intro: Thm: BE}.

	\section{\texorpdfstring{Bounded $\H^{\infty}$-calculus in $\IL^p$ and maximal regularity}{}} \label{Sec: Bounded FC in Lp}
	
	In this section, we complete the proof of Theorem~\ref{Intro: Thm: Bounded FC in Lp}. We assume that $A$ is $p$-elliptic. Let us introduce the relevant terminology, while referring to \cite{Denk-Hieber-Pruess_Memoirs, Haase-book} for background on functional calculus and maximal parabolic regularity.
	
	\begin{definition} \label{Bounded FC in Lp: Def: Dunford-Riesz}
		Let $\omega \in (0, \pi]$. The \textbf{Dunford--Riesz class $\boldsymbol{\H^{\infty}_0(\S_{\omega})}$} consists of all holomorphic functions $\varphi \colon \S_{\omega} \to \C$ with $|\varphi(z)| \leq C (|z|^s \land |z|^{-s})$ for some $C, s> 0$ and all $z \in \S_{\omega}$. 
	\end{definition}
	
	Given $\nicefrac{\pi}{2} - \theta_p < \nu < \omega < \pi$ and $\varphi \in \H_0^{\infty}(\S_{\omega})$, we define $\varphi(\LL_p^A) \in \cL(\IL^p)$ via the Cauchy integral 
	\begin{equation*}
		\varphi(\LL_p^A) \coloneqq \frac{1}{2 \pi \i} \int_{\partial \S_{\nu}} \varphi(z) (z - \LL_p^A)^{-1} \, \d z,
	\end{equation*}
	where the boundary is positively oriented around the spectrum. We note that Lemma~\ref{Lp contractivity: Lem: Lp realization} yields
	\begin{equation} \label{eq: Bounded FC in Lp: FC consistent}
		\varphi(\LL_p^A) f = \varphi(\LL^A) f \qquad (f \in \IL^p \cap \IL^2). 
	\end{equation}
	
	\begin{definition} \label{Bounded FC in Lp: Def: BFC}
		We say that $\LL_p^A$ admits a \textbf{bounded $\boldsymbol{\H^{\infty}}$-calculus of angle $\boldmath{\nicefrac{\pi}{2} - \theta_p}$} if for all $\omega \in (\nicefrac{\pi}{2} - \theta_p, \pi)$ there is some $C > 0$ such that
		\begin{equation} \label{eq: Lp bound FC}
			\| \varphi(\LL_p^A) \|_{\cL(\IL^p)} \leq C \| \varphi \|_{\L^{\infty}(\S_{\omega})} \qquad (\varphi \in \H_0^{\infty}(\S_{\omega})). 
		\end{equation}
	\end{definition}
	
	To prove the second assertion of Theorem~\ref{Intro: Thm: Bounded FC in Lp}, we follow \cite{CD_Open_Set} and first prove the following \emph{weak quadratic estimates} in the spirit of \cite{CDMcIY_BFC}. Since $\IL^p \cap \IL^{p'} \sub \IL^2$, we can avoid $\IL^p$-realizations at this stage. 
	
	\begin{proposition}  \label{Bounded FC in Lp: Prop: Weak QE}
		Let $p \in (1, \infty)$, $A$ be $p$-elliptic and $\theta \in (- \theta_p, \theta_p)$. Then there exists $C > 0$ such that we have
		\begin{equation*}
			\int_0^{\infty} \big| \big \langle t \e^{\i \theta} \LL^A \e^{- t \e^{\i \theta} \LL^A} f,  g\big \rangle \big| \, \frac{\d t}{t} \leq C \Vert f \Vert_p \Vert g \Vert_{p'} \qquad (f,g \in \IL^p \cap \IL^{p'}).
		\end{equation*}
	\end{proposition}
	
	\begin{proof}
		We abbreviate $A(\theta) \coloneqq \e^{\i \theta} A$. Thanks to Lemma~\ref{p-ellipticity: Lem: Basics}~\ref{p-ellipticity: Lem: Basics, duality in A}, we can apply Theorem~\ref{Intro: Thm: BE} to the pair $(A(\theta), A(\theta)^*)$ and get
		\begin{align*}
			\int_0^{\infty} \big| \big \langle t \LL^{A(\theta)} \e^{- 2 t \LL^{A(\theta)}} f, g \big \rangle \big| \, \frac{\d t}{t} &= \int_0^{\infty}  \big| \big \langle \LL^{A(\theta)} f_t^{A(\theta)}, g_t^{A(\theta)^*} \big \rangle \big| \, \d t
			\\&= \int_0^{\infty} \left| \int_O A(\theta) \nabla (f_t^{A(\theta)})|_O \cdot \overline{\nabla (g_t^{A(\theta)^*})|_O} \, \d x \right| \, \d t
			\\&\leq \Lambda(A) \int_0^{\infty} \int_O \big | \nabla (f_t^{A(\theta)})|_O\big| \cdot \big |\nabla (g_t^{A(\theta)^*})|_O\big| \, \d x \, \d t
			\\&\leq C \Lambda(A) \Vert f \Vert_p \Vert g \Vert_{p'}. \qedhere 
		\end{align*}
	\end{proof}
	
	\begin{proof}[\rm\bf{Proof of Theorem~\ref{Intro: Thm: Bounded FC in Lp}, part 2.}]
		We fix $\omega \in (\nicefrac{\pi}{2} - \theta_p, \pi)$. A classical theorem on quadratic estimates \cite[Thm.~4.6 \& Ex.~4.8]{CDMcIY_BFC} states that the estimate in Proposition~\ref{Bounded FC in Lp: Prop: Weak QE} implies that there is $C > 0$ such that
		\begin{equation*}
			\| \varphi(\LL^A) f \|_p \leq C \| \varphi \|_{\L^\infty(\S_{\omega})} \| f \|_p  \qquad (\varphi \in \H_0^{\infty}(\S_{\omega}), f \in \IL^p \cap \IL^{p'}).
		\end{equation*}
        Indeed, we apply the results in \cite{CDMcIY_BFC} to the dual pair $(\IL^p \cap \IL^2, \IL^{p'} \cap \IL^2)$. They require that $\LL^A$ has dense range only for the convergence lemma \cite[p.69]{CDMcIY_BFC}. If $\varphi \in \H_0^{\infty}(\S_{\omega})$, then the latter holds without this extra assumption. By \eqref{eq: Bounded FC in Lp: FC consistent} and density, we obtain \eqref{eq: Lp bound FC}. 
	\end{proof}
	
	Finally, we establish maximal regularity for $\LL_p^A$.
	
	\begin{definition}
		Let $T \in (0, \infty)$. We say that $\LL_p^A$ has \textbf{maximal parabolic regularity} if for some $q \in (1, \infty)$ (or equivalently all $q \in (1, \infty)$) and all $f \in \L^q(0, T; \IL^p)$ the mild solution 
		\begin{equation*}
			u\colon [0,T) \to \IL^p, \quad u(t) = \int_0^t \e^{-(t-s)\LL_p^A}f(s) \, \d s
		\end{equation*}
		to the abstract Cauchy problem
		\begin{align*}
			u'(t) + \LL_p^A u(t) &= f(t) \qquad (t \in (0, T)),\\
			u(0) &=0,
		\end{align*}
		is (Fr\'echet-)differentiable a.e., takes its values in $\D(\LL_p^A)$ a.e.\ and $u', \LL_p^A u$ belong to $\L^q(0, T; \IL^p)$.
	\end{definition}
	
	\begin{proof}[\rm\bf{Proof of Theorem~\ref{Intro: Thm: Bounded FC in Lp}, part 3.}]
		This is just the Dore--Venni theorem \cite[Cor.~9.3.12]{Haase-book} applied to the invertible operator $1+ \LL_p^A$. Indeed, the property of maximal regularity is invariant under shifting the operator and $1+\LL_p^A$ has a bounded $\H^\infty$-calculus of angle $\nicefrac{\pi}{2}-\theta_p$. Therefore, $1 + \LL_p^A$ has bounded imaginary powers~\cite[p.88]{Haase-book}.
		% By \cite[Cor.~2.3.9]{Haase-book} we can replace $\LL_p^A$ by $\LL_p^A|_{\overline{\Ran(\LL_p^A)}}$ in \eqref{eq: Lp bound FC}. Applying the convergence lemma \cite[Prop.~5.1.4]{Haase-book} to the sequence $(\varphi_n(z))_n = (z^{\i s} (z( 1 + z)^{-2})^{\nicefrac{1}{n}})_n$ for $s \in \R$, we deduce that $\LL_p^A|_{\overline{\Ran(\LL_p^A)}}$ has bounded imaginary powers in the sense of \cite[Def.~1]{Pruess-Sohr-BIP}. Thus, a refined version of the Dore--Venni theorem due to Prüss--Sohr \cite[Thm.~6]{Pruess-Sohr-BIP} yields that $\LL_p^A$ admits maximal regularity.
	\end{proof}
	
	%%%%%%%%%%%%%%%%%%%%%%%%%%%%%%%%%%%%%%%%%%%%%%%%%%%%%%%%%%%%%%%%%%%%%%%%%%%%%%%%%%%
	
	\section{Elliptic regularity for dynamical boundary conditions} \label{Sec: Elliptic regularity for DBC}

        In this section, we illustrate how to prove H\"older estimates for elements in the domain of (fractional powers of) $\LL_p^A$ for large $p$. In order to avoid further technicalities, we work in the concrete geometric setting of Section~\ref{Subsec: Concrete geometry} and more specifically assume the following. 

        \begin{assumption} \label{Elliptic regularity for DBC: Ass: Geometry}
		Throughout this section, we suppose that 
		\begin{enumerate}
            \item $O$ is a bounded domain and $V \neq \W^{1,2}_{\varnothing}(O)$. 
			\item $O$ is locally uniform near $\partial O \setminus D$ and there is $\ell \in (d-2,d)$ such that $D$ is an $\ell$-set and $\Sigma$ is an upper $\ell$-set.
		\end{enumerate}
	\end{assumption}

	Boundedness of $O$ implies that $\W^{-1,q}_D(O) \sub V^*$ and $\IL^p \sub \IL^2$ for all $p, q \in [2, \infty)$. By Remark~\ref{Formula for L via L-M: Rem: Coercive a}, the Lax--Milgram operator $\cL \colon V \to V^*$ is invertible. Hence, Lemma~\ref{Formula for L via L-M: Lem: Repres via j} implies that also $\LL^A$ is invertible. Moreover, Theorem~\ref{Traces: Thm: Extension to Lp} yields that $j^* \colon \IL^p \to \W^{-1,q}_D(O)$ is bounded provided that
	\begin{equation} \label{eq: Elliptic regularity for DBC, choice for p,q}
		q \in [2, ((d - \ell) \lor 1)') \quad \& \quad p \in ((\tfrac{\ell q'}{ d-q'})', q] \cap [2, q].
	\end{equation}

        \begin{definition} \label{Elliptic regularity for DBC: Def: Hölder space}
		Let $\mu \in (0,1]$. The \emph{H\"older space $\boldsymbol{\rC^{\mu}(O)}$} consists of all bounded and $\mu$-H\"older continuous functions on $O$ with norm 
		\begin{equation*}
			\| u \|_{\rC^{\mu}(O)} \coloneqq \| u \|_{\L^{\infty}(O)} + \sup_{x, y \in O, x \neq y} \frac{|u(x) - u(y)|}{|x-y|^{\mu}}. 
		\end{equation*}
        The subspace $\rC^{\mu}_D(O)$ consists of all functions in $\rC^{\mu}(O)$ whose unique continuous extension to $\overline{O}$ vanishes in $D$.
	\end{definition}

        Our first result concerns global H\"older regularity to a variational problem with dynamical boundary conditions.
	
	\begin{theorem} \label{Elliptic regularity for DBC: Thm: Hölder through static LM}
		Let $\mu \in (0,1]$ and $p,q$ as in \eqref{eq: Elliptic regularity for DBC, choice for p,q}. If $\cL^{-1} \colon \W^{-1,q}_D(O) \to \rC^{\mu}(O)$ is bounded, then for every $f \in \IL^p$ the unique solution $u$ to the variational problem $\LL^A u = f$ belongs to $\rC^{\mu}_D(O)$. 
	\end{theorem}
	
	\begin{proof}
		The assumption joint with Lemma~\ref{Formula for L via L-M: Lem: Repres via j} and Theorem~\ref{Traces: Thm: Extension to Lp} entail that $$u \coloneqq (\LL^A)^{-1} f  = j \cL^{-1} j^* f \in \rC^{\mu}(O) \cap \D(\LL^A).$$
		Since $D$ is an $\ell$-set for some $\ell \in (d-2,d)$, we can use \cite[Lem.~4.8 \& App.~B]{B-C-E_GE_vs_ER} to conclude that the continuous extension of $u$ to $\overline{O}$ vanishes in $D$.
	\end{proof}

	\begin{remark}
		The assumption on $\cL^{-1}$ in Theorem~\ref{Elliptic regularity for DBC: Thm: Hölder through static LM} is known in a variety of settings, compare with \cite{Auscher-Badr-Haller-Rehberg_Lp_real, Bechtel_Lp, Egert_Riesz-Trafo, Haller-Meinlschmidt-Rehberg_Fractional_Powers}. Interestingly, these references rely on results from harmonic analysis, such as Hölderian Gaussian estimates and the solution of the Kato problem in the `non-dynamical setting' \cite{Kato_Mixed, B-C-E_GE_vs_ER} and we do not know whether these results themselves also hold in our `dynamical setting'.
	\end{remark}

	Theorem~\ref{Elliptic regularity for DBC: Thm: Hölder through static LM} is the starting point for investigating operator theoretic regularity, that is, embeddings of domains of $\IL^p$-realizations into H\"older spaces. The following result is new even for small complex perturbations of real coefficients, so for $p$-elliptic $A$ with very large~$p$.
	
	\begin{corollary} \label{Elliptic regularity for DBC: Cor: dom(L_p) embeds into Hölder}
		Let $\mu \in (0,1]$, $p,q$ as in \eqref{eq: Elliptic regularity for DBC, choice for p,q} and let $A$ be $p$-elliptic. If $\cL^{-1} \colon \W^{-1,q}_D(O) \to \rC^{\mu}(O)$ is bounded, then $\D(\LL_p^A) \sub \rC^{\mu}_D(O)$ with continuous inclusion.
	\end{corollary}
	
	\begin{proof}
        Since $O$ is bounded, $\LL_p^A$ is the part of $\LL^A$ in $\IL^p$, compare with Lemma~\ref{Lp contractivity: Lem: part in Lp}. Hence, if $u \in \D(\LL_p^A)$, then $u = (\LL^A)^{-1} f = j \cL^{-1} j^* f$ with $f \coloneqq \LL_p^A u \in \IL^p$. The claim follows from Theorem~\ref{Elliptic regularity for DBC: Thm: Hölder through static LM} and continuity of the respective linear operators.
	\end{proof}
	
	As pointed out in the introduction, applications to semilinear equations \cite{Disser-Meyries-Rehberg_Parabolic} require similar embedding results already for domains of fractional powers $(\LL_p^A)^{\sigma}$ with $\sigma \in (0,1)$. Fractional powers are defined by functional calculus, see e.g.\ \cite[Chap.~3]{Haase-book}, and all relevant knowledge appears in the proof below. The same goes with interpolation theory and we refer to \cite{Triebel_Int_book} for background information. 
	
	\begin{theorem} \label{Elliptic regularity for DBC: Thm: dom(L_p^{sigma}) embeds into Hölder}
		Let $\mu \in (0,1]$, $p,q$ as in \eqref{eq: Elliptic regularity for DBC, choice for p,q} and $A$ be $p$-elliptic. If $\cL^{-1} \colon \W^{-1,q}_D(O) \to \rC^{\mu}(O)$ is bounded, then
     \begin{equation*}
      \D((\LL_p^A)^{\sigma}) \sub \rC^{\kappa \mu}_D(O)
    \end{equation*}
    for all $\kappa \in (0,1)$ and $\sigma \in (\frac{1}{2} \lor ((1-\kappa)\frac{d}{p \mu + d} + \kappa), 1]$. 
	\end{theorem}
	
	For the proof, we modify an argument in~\cite[Sec.~3]{Haller-Meinlschmidt-Rehberg_Fractional_Powers}. We also need a new interpolation inequality and a compatibility property for domains of fractional powers, the proofs of which we postpone until the end of the section.
	
	\begin{lemma} \label{Elliptic regularity for DBC: Lem: Interpolation inequality}
		Let $p \in [1, \infty]$, $\mu \in (0,1]$ and $\theta \coloneqq \nicefrac{d}{(p \mu + d)}$. Then there is $C > 0$ depending on $d, p$ and geometry such that 
		\begin{equation*}
			\| u \|_{\L^{\infty}(O)} \leq C \| u \|_{\L^p(O)}^{1 - \theta} \| u \|_{\rC^{\mu}(O)}^{\theta} \qquad (u \in \L^p(O) \cap \rC_D^{\mu}(O)).
		\end{equation*}
		In addition, we have for all $\mu \in (0,1]$ and $\theta \in [0,1]$ that 
		\begin{equation*}
			\| u \|_{\rC^{\theta \mu}(O)} \leq 3 \| u \|_{\L^{\infty}(O)}^{1 - \theta} \| u \|_{\rC^{\mu}(O)}^{\theta} \qquad (u \in \rC^{\mu}(O)). 
		\end{equation*}
	\end{lemma}
	
	\begin{lemma}  \label{Elliptic regularity for DBC: Lem: Frac power domain embeds into j(V)}
		If $\sigma > \nicefrac{1}{2}$, then $\D((\LL^A)^\sigma) \subseteq j(V)$.
	\end{lemma}

	\begin{proof}[\rm\bf Proof of Theorem~\ref{Elliptic regularity for DBC: Thm: dom(L_p^{sigma}) embeds into Hölder}.]
		We obtain from \cite[Thm.~1.15.2]{Triebel_Int_book} that
		\begin{align*}
			\D((\LL_p^A)^{\sigma}) &\sub (\IL^p, \D(\LL_p^A))_{\sigma, \infty}.
			\intertext{Since $\sigma > (1 - \kappa) \frac{d}{p \mu + d} + \kappa$, \cite[Thm.~1.3.3~(b) \& (e)]{Triebel_Int_book} entails}
			&\sub (\IL^p, \D(\LL_p^A))_{(1- \kappa) \frac{d}{p \mu + d} + \kappa, 1}.
			\intertext{By the reiteration theorem \cite[Thm.~1.10.2]{Triebel_Int_book} the latter coincides with}
			&= ((\IL^p, \D(\LL_p^A))_{\frac{d}{p \mu + d},1}, \D(\LL_p^A))_{\kappa, 1}.
        \end{align*}
        Now, we apply the bounded operator $|_O$ on both sides and conclude 
        \begin{align*}
			\D((\LL_p^A)^{\sigma})|_O &= ((\L^p(O), \D(\LL_p^A)|_O)_{\frac{d}{p \mu +d},1}, \D(\LL_p^A)|_O)_{\kappa, 1}
             \\&\sub ((\L^p(O), \rC_D^{\mu}(O))_{\frac{d}{p \mu +d},1}, \rC_D^{\mu}(O))_{\kappa, 1},
        \end{align*}
        where the second step is just Corollary~\ref{Elliptic regularity for DBC: Cor: dom(L_p) embeds into Hölder}. In the language of \cite[Lem.~1.10.1~(a)]{Triebel_Int_book}, the first statement of Lemma~\ref{Elliptic regularity for DBC: Lem: Interpolation inequality} means  that $\L^{\infty}(O)$ is of class $J(\nicefrac{d}{(p \mu + d)})$ between $\L^p(O)$ and $\rC_D^{\mu}(O)$ and hence contains $(\L^p(O), \rC_D^{\mu}(O))_{\nicefrac{d}{(p \mu + d)}, 1}$ with continuous inclusion. A similar reformulation applies to the second statement of the lemma. We deduce 
		\begin{equation*}
			\D((\LL_p^A)^{\sigma})|_O \sub (\L^{\infty}(O), \rC_D^{\mu}(O))_{\kappa,1} \sub \rC_D^{\kappa \mu}(O).
		\end{equation*}
		Now, let $u \in \D((\LL_p^A)^{\sigma})$. We already know that $u|_O \in \rC^{\kappa \mu}_D(O)$. Let $\widetilde{u} \in \rC^{\kappa \mu}_D(O)$ be its unique continuous extension. We have $\D((\LL_p^A)^{\sigma}) \sub \D((\LL^A)^{\sigma})$, see \cite[Prop.~2.6.5~b)]{Haase-book}, and therefore Lemma~\ref{Elliptic regularity for DBC: Lem: Frac power domain embeds into j(V)} yields $u \in j(V)$. We conclude that $u|_{\Sigma} = \tr(u|_O) = \widetilde{u}|_{\Sigma}$ and hence $u = \widetilde{u} \in \rC^{\kappa \mu}_D(O)$.
	\end{proof}
	
	\begin{proof}[\rm\bf Proof of Lemma~\ref{Elliptic regularity for DBC: Lem: Interpolation inequality}.]
		The second assertion has been proven in \cite[Lem.~3.5]{Haller-Meinlschmidt-Rehberg_Fractional_Powers}. Let us turn to the first part. To this end, we fix $\delta \in (0,1]$ as in the definition of a locally uniform domain near $\partial O \setminus D$, see \cite[Def.~2.3]{Kato_Mixed}. If $\| u \|_{\L^p(O)} \geq (\frac{\delta}{2})^{\nicefrac{d}{\theta p}} \| u \|_{\rC^{\mu}(O)}$, then we simply estimate 
		\begin{equation*}
			\| u \|_{\L^{\infty}(O)} = \| u \|_{\L^{\infty}(O)}^{1 - \theta} \| u \|_{\L^{\infty}(O)}^{\theta} \leq (\tfrac{2}{\delta})^{\frac{d(1 - \theta)}{\theta p}} \| u \|_{\L^p(O)}^{1 - \theta} \| u \|_{\rC^{\mu}(O)}^{\theta}.  
		\end{equation*}
		Hence, we assume from now on that $\| u \|_{\L^p(O)} < (\frac{\delta}{2})^{\nicefrac{d}{\theta p}}  \| u \|_{\rC^{\mu}(O)}$ and set 
		\begin{equation*}
			r \coloneqq (\| u \|_{\L^p(O)} \| u \|_{\rC^{\mu}(O)}^{-1})^{\frac{\theta p}{d}} \in (0,\tfrac{\delta}{2}). 
		\end{equation*} 
		Let $x \in O$. We consider two cases. 
		
		\textbf{(1) $\boldsymbol{\overline{B(x,r)} \cap D \neq \varnothing}$.} We fix $x_D \in \overline{B(x,r)} \cap D$. Since $u$ vanishes in $D$ we get 
		\begin{equation*}
			|u(x)| = |u(x) - u(x_D)| \leq r^{\mu} \| u \|_{\rC^{\mu}(O)} \leq \| u \|_{\L^p(O)}^{1 - \theta} \| u \|_{\rC^{\mu}(O)}^{\theta}. 
		\end{equation*}
		\textbf{(2)  $\boldsymbol{\overline{B(x,r)} \cap D = \varnothing}$.} Then $\overline{B(x,r)} \sub O$ or $\overline{B(x,r)} \cap (\partial O \setminus D) \neq \varnothing$ and we obtain for all $y \in O \cap B(x,r)$ that 
		\begin{equation*}
			|u(x)| \leq |u(x) - u(y)| + |u(y)| \leq r^{\mu} \| u \|_{\rC^{\mu}(O)} + |u(y)|. 
		\end{equation*}
		We average the latter with respect to $y \in O \cap B(x,r)$ and use Hölder's inequality, which leads us to 
		\begin{equation*}
			|u(x)| \leq r^{\mu} \| u \|_{\rC^{\mu}(O)} + |O \cap B(x,r) |^{- \frac{1}{p}} \| u \|_{\L^p(O)}. 
		\end{equation*}
		If $\overline{B(x,r)} \sub O$, then $|O \cap B(x,r)| = |B(0,1)| r^d$. In the other case, $\dist(x, \partial O \setminus D) \leq r < (\nicefrac{\delta}{2} \land 1)$. It follows from \cite[Prop.~2.9]{Kato_Mixed} that there is $C > 0$ depending on $d$ and geometry such that $|O \cap B(x,r)| \geq C r^d$. We conclude that  
		\begin{equation*}
			\| u \|_{\L^{\infty}(O)} \leq (1 + C^{- \frac{1}{p}}) \| u \|_{\L^p(O)}^{1 - \theta}  \| u \|_{\rC^{\mu}(O)}^{\theta}.  \qedhere 
		\end{equation*}
	\end{proof}
	
	\begin{proof}[\rm\bf Proof of Lemma~\ref{Elliptic regularity for DBC: Lem: Frac power domain embeds into j(V)}.]
		By \cite[Prop.~3.1.9~a)]{Haase-book} we have $\D((\LL^A)^{\sigma}) = \D((1 + \LL^A)^{\sigma})$. Since $1 + \LL^A$ is invertible, so are its fractional powers \cite[Prop.~3.1.1~e)]{Haase-book}. We let $u \in \D((1 + \LL^A)^{\sigma})$ and put $f \coloneqq (1 + \LL^A)^{\sigma} u$. We have the following formula \cite[Cor.~3.3.6]{Haase-book}: 
		\begin{equation*}
			u = (1 + \LL^A)^{-\sigma} f = \frac{1}{\Gamma(\sigma)} \int_0^{\infty} t^{\sigma} \e^{-t} f_t^A \, \frac{\d t}{t} = \frac{1}{\Gamma(\sigma)} \int_0^{\infty} t^{\sigma} \e^{-t} j( (f_t^A)|_O )\, \frac{\d t}{t},  
		\end{equation*} 
        where in the last step we have used that $f_t^A \in \D(\LL^A) \sub j(V)$. Since $j \colon V \to \IL^2$ is bounded, the assertion follows once we have checked that 
		\begin{equation*}
			\int_0^{\infty} t^{\sigma} \e^{-t} (f_t^A)|_O \, \frac{\d t}{t}
		\end{equation*}
		converges absolutely in $V$. Due to the exponential decay and the assumption $\sigma > \nicefrac{1}{2}$, it suffices to show that there is $C > 0$ such that $\| f_t^A|_O \|_V \leq C t^{-\nicefrac{1}{2}}$ for all $t > 0$. In order to see this, we use the coercivity of $a$ and the Cauchy--Schwarz inequality and get 
		\begin{equation*}
			 C \| f_t^A|_O \|_V^2 \leq \re a(f_t^A|_O, f_t^A|_O) = \re \langle \LL^A f_t^A, f_t^A \rangle_2 \leq \| \LL^A f_t^A \|_2 \| f_t^A \|_2. 
		\end{equation*}
		Since $(\e^{-t \LL^A})_{t \geq 0}$ is a bounded analytic semigroup in $\IL^2$, the right-hand side is bounded from above by $C t^{-1} \| f \|_2^2$. This completes the proof. 
	\end{proof}
	%%%%%%%%%%%%%%%%%%%%%%%%%%%%%%%%%%%%%%%%%%%%%%%%%%%%%%%%%%%%%%%%%%%%%%%%%%%%%%%%%%%%%%%%%%%%%%%%%%%%
	
	\section{Explicit angle} \label{Sec: Explicit angle}
	
	We close this paper by elaborating a quantitative lower bound for $\theta_p$, depending on $p$ and the data $\Lambda(A)$, $\lambda(A)$ and $\omega(A)$. Under additional symmetry assumptions on~$A$, results of this type already exist, see Remark~\ref{Explicit angle: Rem: Known results}. For $p \in (1, \infty) $ we set
	\begin{equation*}
		\sigma_p \coloneqq \frac{|p-2|}{2 \sqrt{p-1}}
	\end{equation*}
    and note that $\sigma_p = \sigma_{p'}$. We will use the fact that (quantitative) smallness of $\| \im(A) \|_{\infty}$ implies $p$-ellipticity of $A$. 

    \begin{lemma}[{\cite[Lem.~2.16]{P-ellipticity-Moritz-Counterpart}}]
        \label{Explicit lower bound for theta_p: Lem: Criterion}
        Let $p \in (2, \infty)$ and assume that $\sigma_p \| \im(A) \|_{\infty} < \lambda(A)$. Then $A$ is $p$-elliptic.
    \end{lemma}
 
 Hence, our goal is to find the largest possible $\theta$ such that $\e^{\pm \i \theta} A$ satisfies this smallness assumption. The method cannot be optimal, because the same sesquilinear form $a$ may be represented by a variety of matrices $A$ with different size of $\| \im(A) \|_{\infty}$, see \cite[Ex.~1]{CM_p-elliptic}. On the other hand, it does not require any symmetry for $A$.
	
	\begin{proposition} \label{Explicit lower bound for theta_p: Prop: Main result}
		Let $p \in (1, \infty)$ and assume that $\sigma_p \| \im(A) \|_{\infty} < \lambda(A)$. Then
		\begin{equation*}
			\tan(\theta_p) \geq \frac{\lambda(A) - \sigma_p \| \im(A) \|_{\infty}}{\lambda(A) \tan(\omega(A)) + \sigma_p \| \re(A) \|_{\infty}}.
		\end{equation*}
	\end{proposition}
	
	\begin{proof}
		By symmetry and Lemma~\ref{p-ellipticity: Lem: Basics}~\ref{p-ellipticity: Lem: Basics, duality in p} we can assume $p>2$. We fix $\theta \in (0, \nicefrac{\pi}{2}-\omega(A))$. We need to prove $\Delta_p(\e^{\pm \i \theta} A) > 0$ provided that
        \begin{equation} \label{eq: Lower bound for theta_p}
			\tan(\theta) < \frac{\lambda(A) - \sigma_p \| \im(A) \|_{\infty}}{\lambda(A) \tan(\omega(A)) + \sigma_p \| \re(A) \|_{\infty}}.
	\end{equation}
  
  Lemma~\ref{Explicit lower bound for theta_p: Lem: Criterion} says that $\Delta_p(\e^{\i \theta} A)>0$ if
		\begin{equation*}
			\sigma_p \| \im(\e^{\i \theta} A ) \|_{\infty} < \lambda(\e^{\i \theta} A).
		\end{equation*}
		We have the rough estimate
		\begin{equation*}
			\| \im(\e^{\i \theta} A) \| = \|  \sin(\theta) \re(A) + \cos(\theta) \im(A) \|_{\infty} \leq \sin(\theta) \| \re(A) \|_{\infty} + \cos(\theta) \| \im(A) \|_{\infty}
		\end{equation*}
		and, by rotating the set $\{z\in \C \colon |\arg(z)| \leq \omega(A) \, \& \, \re(z) \geq \lambda\}$ and finding the point with smallest real part, we get
		\begin{equation*}
			\lambda(\e^{\i \theta} A) \geq \lambda(A) \frac{\cos(\omega(A) + \theta)}{\cos(\omega(A))} = \lambda(A) \big(\cos(\theta) - \sin(\theta) \tan(\omega(A)) \big).
		\end{equation*}
		So, a sufficient criterion for $\Delta_p(\e^{\i \theta} A)>0$ is
		\begin{equation*}
			\sigma_p \big( \sin(\theta) \| \re(A) \|_{\infty} + \cos(\theta) \| \im(A) \|_{\infty} \big) < \lambda(A) \big( \cos(\theta) - \sin(\theta) \tan(\omega(A)) \big).
		\end{equation*}
		We divide both sides by $\cos(\theta) > 0$ and rearrange terms to get
		\begin{equation*}
			\tan(\theta) \big( \lambda(A) \tan(\omega(A)) + \sigma_p \| \re(A) \|_{\infty} \big) < \lambda(A) - \sigma_p \| \im(A) \|_{\infty},
		\end{equation*}
		which is equivalent to \eqref{eq: Lower bound for theta_p}. 
  
    Since the right-hand side in \eqref{eq: Lower bound for theta_p} stays the same when replacing $A$ by $A^*$, the restriction on $\theta$ in \eqref{eq: Lower bound for theta_p} also implies $\Delta_p(\e^{\i \theta} A^*) > 0$ and hence $\Delta_p(\e^{- \i \theta} A) > 0$, see Lemma~\ref{p-ellipticity: Lem: Basics}~\ref{p-ellipticity: Lem: Basics, duality in A}. This completes the proof.
	\end{proof}
	
	\begin{remark} \label{Explicit angle: Rem: Known results}
		The following results are known from the literature.
		
		\begin{enumerate}
			\item For real-valued $A$, \cite[Thm.~3.4]{Chill_Lp-angle} gives an upper bound for the angle of the numerical range of $\LL_p^A$. This leads to the better bound
			\begin{equation*}
				\tan(\theta_p) \geq \frac{1}{\sqrt{\sigma_p^2 + \frac{p^2}{4 (p-1)} \tan(\omega(A))^2}}.
			\end{equation*}
			The proof uses the fact that the coefficients are real and there is no simple adaptation in the complex case.
			
			\item Cialdea and Maz'ya characterize in \cite[Thm.~9]{CM_optimal angle} the optimal angle of the $\L^p$-dissipativity under Dirichlet boundary conditions by an implicit formula. For symmetric $A$ (that is $A = A^t$ in their terminology) the expression simplifies, see \cite[Thm.~1]{CM_optimal angle}.
			
			\item If $\im(A)$ is symmetric and $|1 - \nicefrac{2}{p}| < \cos(\omega(A))$, Do found in \cite[Thm.~1.1]{Do_p-ellipticity} a lower bound for $\theta_p$ in terms of $\omega(A)$ and $p$:
			\begin{align*}
				\theta_p &\geq \arccos(|1 - \nicefrac{2}{p}|) - \omega(A) \qquad &(\re(A) \; \text{symmetric}), \\
				\tan(\theta_p) &\geq \frac{1 - \tan(\omega(A)) \sigma_p}{\frac{p - \sqrt{p-1}}{\sqrt{p-1}} \tan(\omega(A)) + \sigma_p} \qquad &(\re(A) \; \text{not symmetric}).
			\end{align*}
			This bound does not recover the bound in (i) when $A$ is real.
		\end{enumerate}
	\end{remark}
	%%%%%%%%%%%%%%%%%%%%%%%%%%%%%%%%%%%%%%%%%%%%%%%%%%%%%%%%%%%%%%%%%%%%%%%%%%%%%%%%%%%%%%%%%%%%%%%%%%%%%%%%%%%%%%%%%%%%%%%%%%%%%%%%%%%%%%%%%%%%%%%%%%%%%%%%%%%%%%%%%%%%%%%%%%%%%%%%%%%%%%%%%%%%%%%%%%%%%%%%%%%%

\end{document}